\documentclass[11pt,reqno]{amsart}
\usepackage{graphicx}
\usepackage{float}
\usepackage[caption = false]{subfig}
\usepackage{enumerate}
\usepackage{mathtools}
\usepackage{breqn}
\usepackage{array, geometry, graphicx}
\usepackage{amsmath,amsfonts,amssymb,amsthm}
\usepackage{multirow}
\newtheorem{theorem}{Theorem}[section]

\newtheorem{lemma}[theorem]{Lemma}

\theoremstyle{definition}

\theoremstyle{remark}

\numberwithin{equation}{section}
\DeclareMathOperator{\RE}{Re}

\begin{document}

\title{On sharp bounds of certain Close-to-Convex functions}
\thanks{The first author is supported by The Council of Scientific and Industrial Research(CSIR). Ref.No.:08/133(0018)/2017-EMR-I. The second author is supported by Delhi Technological University under the project with reference no. DTU/IRD/619/2019/2106.}
\author{Priyanka Goel}
\address{Department of Applied Mathematics, Delhi Technological University, Delhi--110042, India}
\email{priyanka.goel0707@gmail.com}
\author[S. Sivaprasad Kumar]{S. Sivaprasad Kumar}
\address{Department of Applied Mathematics, Delhi Technological University, Delhi--110042, India}
\email{spkumar@dce.ac.in}


\begin{abstract}
We derive general formula for the fourth coefficient of the functions belonging to the Carath\'{e}odory class involving the parameters lying in the open unit disk. Further, we obtain sharp upper bounds of initial inverse coefficients  for certain close-to-convex functions satisfying any one of the inequalities: $\RE((1-z)f'(z))>0,$ $\RE((1-z^2)f'(z))>0,$ $\RE((1-z+z^2)f'(z))>0$ and $\RE((1-z)^2f'(z))>0$.
\end{abstract}

\maketitle

\section{Introduction}
Let $\mathcal{A}$ denote the class of functions of the form
\begin{equation}\label{fz}
  f(z)=z+\sum_{n=2}^{\infty}a_nz^n=z+a_2z^2+a_3z^3+\ldots,
\end{equation}
which are analytic in the open unit disk $\mathbb{D}=\{z\in\mathbb{C}:|z|<1\}.$ Let $\mathcal{S}$ be a subclass of $\mathcal{A}$ consisting of all univalent (one-to-one) functions in $\mathcal{A}.$ A function $f\in\mathcal{A}$ is said to be starlike if $f$ maps $\mathbb{D}$ onto a domain which is starlike with respect to origin. The class of starlike functions in $\mathcal{S}$ is denoted by $\mathcal{S^*}$ and is analytically characterized as $\RE(zf'(z)/f(z))>0$ in $\mathbb{D}$. Similarly, a function $f\in\mathcal{A}$ is said to be convex if it maps $\mathbb{D}$ onto a domain which is convex. The class of convex functions in $\mathcal{S}$ is denoted by $\mathcal{C}$. We say that $f\in\mathcal{C}$ if and only if $\RE((1+zf''(z))/f'(z))>0$ in $\mathbb{D}.$ Further, Alexander theorem establishes a two way passage between $\mathcal{S^*}$ and $\mathcal{C}$ namely, $f\in\mathcal{C}$ if and only if $zf'\in\mathcal{S^*}$. A function $f$ defined on $\mathbb{D}$ is said to be close-to-convex with respect to a starlike function $g$ and with argument $\alpha \in(-\pi/2,\pi/2),$ if $\RE(e^{i\alpha}zf'(z)/g(z))>0$ and the class of all such functions is denoted by $\mathcal{K}_{\alpha}(g)$. Let $\mathcal{K}(g)$ be the class of close-to-convex functions with respect to $g$ and $\mathcal{K}_{\alpha}$ be the class of close-to-convex functions with argument $\alpha$ and are given by
 \begin{equation*}
   \mathcal{K}(g):=\bigcup_{\alpha\in(-\pi/2,\pi/2)}\mathcal{K}_{\alpha}(g)\qquad\text{and}\qquad\mathcal{K}_{\alpha}:=\bigcup_{g\in\mathcal{S^*}}\mathcal{K}_{\alpha}(g)
\end{equation*}
respectively. The class $\mathcal{K}$ defined as
\begin{equation*}
  \mathcal{K}:=\bigcup_{\alpha\in(-\pi/2,\pi/2)}\mathcal{K}_{\alpha}=\bigcup_{g\in\mathcal{S^*}}\mathcal{K}(g),
\end{equation*}
denote the class of close-to-convex functions. Kaplan~\cite{kaplan} proved that every close-to-convex function is univalent in $\mathbb{D}.$ In geometrical terms, it means that if $f\in\mathcal{K}$, then the complement of the image of $\mathbb{D}$ under $f$ is the union of non-intersecting half lines. Let $\mathcal{P}$ denote the class of analytic functions of the form
\begin{equation}\label{cartheodry}
  p(z)=1+\sum_{n=1}^{\infty}c_nz^n=1+c_1z+c_2z^2+c_3z^3+\cdots \quad (z\in \mathbb{D}),
\end{equation}
such that $\RE p(z)>0$ and this class is known as Carath\'{e}odory class. To prove our main results, we need the following results pertaining to the class $\mathcal{P}$:
\begin{lemma}\emph{\cite[p.41]{pomm}}\label{lemma3}
If $p(z)$ is in $\mathcal{P}$ and is given by~\eqref{cartheodry}, then $|c_n|\leq 2$ for each $n$.
\end{lemma}
\begin{lemma}\emph{\cite{pomm}}\label{lemma1}
  Let $p\in\mathcal{P}$ and is given by~\eqref{cartheodry}. Then
  \begin{equation*}
    \left|c_2-\dfrac{c_1^2}{2}\right|\leq2-\dfrac{|c_1|^2}{2}.
  \end{equation*}
  This inequality is sharp for the functions $P_{t,\vartheta}(z)$ given by
  \begin{equation*}
    P_{t,\vartheta}(z)=t\left(\dfrac{1+e^{i\vartheta}z}{1-e^{i\vartheta}z}\right)+(1-t)\left(\dfrac{1+e^{i2\vartheta}z^2}{1-e^{i2\vartheta}z^2}\right),
  \end{equation*}
  $0\leq t\leq 1$ and $0\leq\vartheta\leq2\pi.$
\end{lemma}
\begin{lemma}\emph{\cite{earlycoef}}\label{lemma5}
   Let $p\in\mathcal{P}$ and is given by~\eqref{cartheodry}. Then
   $$|c_3-2c_1c_2+c_1^3|\leq 2$$
   and
   $$|c_1^4-3c_1^2c_2+c_2^2+2c_1c_3-c_4|\leq 2.$$
   This result is sharp for the function $p(z)=(1+z)/(1-z).$
\end{lemma}
\begin{lemma}\emph{\cite{max}}\label{maxlemma}
Let
\begin{equation}\label{Omega}
\Omega(A,B,C,M)=\max_{|v|\leq1}(|M|(1-|v|^2)+|A+Bv+Cv^2|).
\end{equation}
If $AC\geq 0,$ then
\begin{equation*}
  \Omega(A,B,C,M)=\begin{cases}
   & |A|+|B|+|C|,\hspace{2cm} |B|\geq 2(|M|-|C|),\\
   & |M|+|A|+\dfrac{B^2}{4(|M|-|C|)}, \quad |B|< 2(|M|-|C|).
     \end{cases}
\end{equation*}
If $AC<0,$ then
\begin{equation*}
\Omega(A,B,C,M)=\begin{cases}
 & 1-|A|+\dfrac{B^2}{4(|M|-|C|)},\quad -4AC(M^2-C^2)\leq B^2C^2\wedge |B|<2(|M|-|C|),\\
 & 1+|A|+\dfrac{B^2}{4(|M|+|C|)},\quad B^2<\min\{4(|M|+|C|)^2,-4AC(M^2C^{-2}-1)\},\\
 & S(A,B,C),\hspace{2.5cm}otherwise,
\end{cases}
\end{equation*}
where
\begin{equation*}
  S(A,B,C)=\begin{cases}
             |A|+|B|-|C|, & \mbox{if }\; |C|(|B|+4|A|)\leq |AB|, \\
             -|A|+|B|+|C|, & \mbox{if }\; |AB|\leq |C|(|B|-4|A|), \\
             (|C|+|A|)\sqrt{1-\dfrac{B^2}{4AC}},\; & \mbox{otherwise}.
           \end{cases}
\end{equation*}
\end{lemma}
It may be noted that the function given by
\begin{equation}\label{H}
H_{t,\beta}(z):=(1-2t)\left(\dfrac{1+z}{1-z}\right)+t\left(\dfrac{1+\beta z}{1-\beta z}\right)+t\left(\dfrac{1+\bar{\beta}z}{1-\bar{\beta}z}\right)
\end{equation}
is a Carath\'{e}odary function if $t\in[0,1/2]$ and $|\beta|=1$, we make use of it to find the extremal functions needed for sharp bounds. Let the inverse of $f(z)$ defined on a disk of radius at least $1/4$ has a series expansion of the form
\begin{equation}\label{finv}
  f^{-1}(w)=w+\sum_{n=2}^{\infty}\delta_nw^n=w+\delta_2w^2+\delta_3w^3+\cdots.
\end{equation}
Since $z=f^{-1}(w),$ we have from~\eqref{fz}
\begin{equation}\label{k1}
  w=f^{-1}(w)+a_2(f^{-1}(w))^2+a_3(f^{-1}(w))^3+\cdots .
\end{equation}
Comparing the coefficients on both sides of~\eqref{k1} after replacing the value of $f^{-1}(w)$ given in~\eqref{finv}, we obtain the following relations:
\begin{eqnarray}\label{deltanai1}
  \delta_2 &=& -a_2 \\ \label{deltanai2}
  \delta_3 &=& 2a_2^2-a_3 \\ \label{deltanai3}
  \delta_4 &=& 5a_2a_3-5a_2^3-a_4\\ \label{deltanai4}
  \delta_5 &=& 14a_2^4-21a_2^2a_3+6a_2a_4+3a_3^2-a_5.
\end{eqnarray}
In 1982, Libera and Z\l otkiewicz~\cite{earlycoef} derived the bounds for the magnitude of the first seven inverse coefficients $\delta_k\;(k=2,3...7)$ of the functions in $\mathcal{C}.$ Later in 1984, they estimated the bounds for first six inverse coefficients of the functions, whose derivative belongs to $\mathcal{P}$(see~\cite{coeffpos}). In 1989, Silverman estimated the bounds of the inverse coefficients for starlike functions (see~\cite{silver}). Similar work has been carried out for various other classes such as class of starlike functions of positive order~\cite{kapoor}, spirallike functions~\cite{spiralike}, starlike functions represented by symmetric gap series~\cite{symmgap} and strongly starlike functions~\cite{rmali}. In 1987, Jenkins~\cite{jenkin} proved that any function in $\mathcal{S}$ having integer coefficients is one of the following nine functions:
 $$z,\;\dfrac{z}{1\pm z},\;\dfrac{z}{1\pm z^2},\;\dfrac{z}{(1\pm z)^2},\;\dfrac{z}{1\pm z+z^2}.$$
 Clearly by taking $g(z)=z,$ the class $\mathcal{K}(g)$ reduces to the class of functions whose derivative belongs to $\mathcal{P}$ which was already considered in~\cite{coeffpos}. Also, the functions appearing in each of the pairs differs by the sign of the coefficients and thus does not affect the bounds. Therefore, in the present paper, we consider the following four subclasses of close-to-convex functions defined as:
\begin{eqnarray*}
  \mathcal{F}_1 &:=& \{f\in\mathcal{A}:\RE(1-z)f'(z)>0\;\text{for}\;z\in\mathbb{D}\} \\
  \mathcal{F}_2 &:=& \{f\in\mathcal{A}:\RE(1-z^2)f'(z)>0\;\text{for}\;z\in\mathbb{D}\} \\
  \mathcal{F}_3 &:=& \{f\in\mathcal{A}:\RE(1-z+z^2)f'(z)>0\;\text{for}\;z\in\mathbb{D}\}\\
  \mathcal{F}_4 &:=& \{f\in\mathcal{A}:\RE(1-z)^2f'(z)>0\;\text{for}\;z\in\mathbb{D}\}.
\end{eqnarray*}
It is pertinent to mention here that Ponnusamy et\;al.~\cite{f1f2} generalized the classes $\mathcal{F}_1$ and $\mathcal{F}_2$ to a class of harmonic close-to-convex functions defined on $\mathbb{D}.$ In~\cite{mainvasu}, Kumar and Vasudevarao estimated the bounds of first three logarithmic coefficients of functions belonging to $\mathcal{F}_1,\mathcal{F}_2$ and $\mathcal{F}_3.$ We observe from the relations~\eqref{deltanai1}-\eqref{deltanai4}  that the inverse coefficients can be expressed in terms of $a_i$'s, which further can be represented in Carath\'{e}odory coefficients. The computation of the upper bound of $|\delta_4|$ requires parametric formulas for second and third Carath\'{e}odory coefficients. We know that the classes $\mathcal{F}_i (i=1,2,3,4)$ are not rotation invariant. Thus in order to avoid the assumption $c_1\geq0$, we use the general formula for $c_3,$ given by Cho et al.~\cite{cho} as follows:
\begin{lemma}\emph{\cite{cho}}\label{lemmacho}
  If $p\in\mathcal{P}$ is of the form $p(z)=1+c_1z+c_2z^2+c_3z^3+\cdots,$ then
  \begin{eqnarray*}
    c_1 &=& 2\zeta_1 \\
    c_2 &=& 2\zeta_1^2+2(1-|\zeta_1|^2)\zeta_2 \\
    c_3 &=& 2\zeta_1^3+4(1-|\zeta_1|^2)\zeta_1\zeta_2-2(1-|\zeta_1|^2)\bar{\zeta_1}\zeta_2^2+2(1-|\zeta_1|^2)(1-|\zeta_2|^2)\zeta_3
    \end{eqnarray*}
    for some $\zeta_i\in\overline{\mathbb{D}}\;(i=1,2,3).$
\end{lemma}
In view of the same, in the following section we derive the general formula for the fourth Carath\'{e}odory coefficient which is required for the estimation of the upper bound of $|\delta_5|.$
\section{The fourth coefficient}
In 2017, Kwon et al.~\cite{kwon} derived a formula for the coefficient $c_4$ so that it can be expressed in terms of $c_1$ and other general variables with the assumption $c_1\geq 0.$ However, the formulas for $c_2$ and $c_3$ were proved way back in 1982 by Libera and Z\l otkiewicz~\cite{earlycoef} with the same assumption. In 2018, Cho et al.~\cite{cho} mentioned that the formula for $c_3$ cannot work for the classes which are not rotation invariant due to the assumption $c_1\geq0$ and hence gave a general formula for $c_3.$ Now in the following result we prove the general formula for $c_4$ which will be useful for other coefficient problems for the classes that are not rotation invariant. We present two different methods of proof. Although, the first method is similar to the one given by Kwon et al.~\cite{kwon} but it involves tedious computations. The second one involves construction of the extremal function as done by Cho et al. in~\cite{cho}.
  \begin{lemma}\label{ownlem3}
    If $p\in\mathcal{P}$ is of the form $p(z)=1+c_1z+c_2z^2+c_3z^3+\cdots,$ then
    \begin{eqnarray}\label{state}
    \nonumber  \dfrac{1}{2}c_4&=&\zeta_1^4+(1-|\zeta_1|^2)(3\zeta_1^2\zeta_2-2\zeta_2^2|\zeta_1|^2+\bar{\zeta_1}^2\zeta_2^3)+\zeta_2^2(1-|\zeta_1|^2)^2+(1-|\zeta_1|^2)(1-|\zeta_2|^2)(2\zeta_1\zeta_3\\
      & &-2\bar{\zeta_1}\zeta_2\zeta_3-\bar{\zeta_2}\zeta_3^2+(1-|\zeta_3|^2)\zeta_4)
    \end{eqnarray}
    for some $\zeta_i\in\overline{\mathbb{D}}.$
  \end{lemma}
  \begin{proof}
Using the fact that the Toeplitz determinants of the series of a function in $\mathcal{P}$ are non-negative, we have
$$D_4=\left|
  \begin{array}{ccccc}
    2 & c_1 & c_2 & c_3 & c_4\\
    \bar{c_1} & 2 & c_1 & c_2  & c_3\\
    \bar{c_2} & \bar{c_1} & 2 & c_1 & c_2 \\
    \bar{c_3} & \bar{c_2} & \bar{c_1}  & 2 & c_1\\
    \bar{c_4} & \bar{c_3} & \bar{c_2} & \bar{c_1} & 2\\
  \end{array}
\right|\geq 0,$$
which is equivalent to
\begin{equation*}
  -2M_1+c_1M_2-c_2M_3+c_3M_4-c_4M_5\leq 0,
\end{equation*}
where $M_i\;(i=1,2,3,4,5)$ is the determinant of the minor matrix of the ith entry in the first row in $D_4.$ The above inequality can be further written as
\begin{eqnarray*}
& & -2M_1+|c_1|^2M_2^1-c_1\bar{c_2}M_2^2+c_1\bar{c_3}M_2^3-c_1\bar{c_4}M_2^4-c_2\bar{c_1}M_3^1+|c_2|^2M_3^2-c_2\bar{c_3}M_3^3+c_2\bar{c_4}M_3^4\\
& &+c_3\bar{c_1}M_4^1-c_3\bar{c_2}M_4^2+|c_3|^2M_4^3-c_3\bar{c_4}M_4^4-c_4\bar{c_1}M_5^1+c_3\bar{c_2}M_5^2-c_4\bar{c_3}M_5^3+|c_4|^2M_5^4\leq 0,
\end{eqnarray*}
where $M_i^j\;(i=1,2,3,4,5,j=1,2,3,4)$ is the determinant of the minor matrix of the jth entry in the first column in $M_i.$ It is easy to verify that $M_5^1=\overline{M}_2^4,$ $M_5^2=\overline{M}_3^4$ and $M_5^3=\overline{M}_4^4.$ So the inequality above reduces to
\begin{eqnarray*}
& & -c_1\bar{c_4}M_2^4+c_2\bar{c_4}M_3^4-c_3\bar{c_4}M_4^4-c_4\bar{c_1}\overline{M}_2^4+c_4\bar{c_2}\overline{M}_3^4-\bar{c_3}c_4\overline{M}_4^4+|c_4|^2M_5^4-2M_1+|c_1|^2M_2^1\\
& &-c_1\bar{c_2}M_2^2+c_1\bar{c_3}M_2^3 -\bar{c_1}c_2M_3^1+|c_2|^2M_3^2-c_2\bar{c_3}M_3^3+\bar{c_1}c_3M_4^1-\bar{c_2}c_3M_4^2+|c_3|^2M_4^3\leq 0.
\end{eqnarray*}
We know that $D_2=M_5^4\geq 0,$ so we multiply both sides of the above inequality by $M_5^4$ and obtain
\begin{eqnarray}\label{1}
  \nonumber |c_4M_5^4-c_1M_2^4+c_2M_3^4-c_3M_4^4|^2-|-c_1M_2^4+c_2M_3^4-c_3M_4^4|^2\leq M_5^4(2M_1-|c_1|^2M_2^1\\
  +c_1\bar{c_2}M_2^2-c_1\bar{c_3}M_2^3+\bar{c_1}c_2M_3^1-|c_2|^2M_3^2+c_2\bar{c_3}M_3^3-\bar{c_1}c_3M_4^1+\bar{c_2}c_3M_4^2-|c_3|^2M_4^3).
\end{eqnarray}
Let us take
\begin{eqnarray}\label{k4}
\nonumber B_1&:=&|-c_1M_2^4+c_2M_3^4-c_3M_4^4|^2+ M_5^4(2M_1-|c_1|^2M_2^1+c_1\bar{c_2}M_2^2-c_1\bar{c_3}M_2^3+\bar{c_1}c_2M_3^1\\
  & & -|c_2|^2M_3^2+c_2\bar{c_3}M_3^3-\bar{c_1}c_3M_4^1+\bar{c_2}c_3M_4^2-|c_3|^2M_4^3),
\end{eqnarray}
which after substituting the values of $M_i^j$'s can be written as
\begin{equation}\label{k2}
B_1=B^2,
\end{equation}
where
\begin{eqnarray}\label{B}
 \nonumber  B&=&16-12|c_1|^2+|c_1|^4+4\bar{c_1}^2c_2+c_3\bar{c_1}^3+(4c_1^2-2|c_1|^2c_2-8c_2+4\bar{c_1}c_3)\bar{c_2}+(c_2^2-c_1c_3)\bar{c_2}^2\\
  & &-(c_1^3-4c_1c_2+4c_3+\bar{c_1}c_2^2-|c_1|^2c_3)\bar{c_3}.
\end{eqnarray}
Now let us take
\begin{equation}\label{k3}
  A=c_1M_2^4-c_2M_3^4+c_3M_4^4,
\end{equation}
which can be further written as
\begin{eqnarray}\label{A}\nonumber
 A&=&c_1^2+4c_2^2+8c_1c_3-4\bar{c_1}c_2c_3-2|c_1|^2c_1c_3-6c_1^2c_2+2|c_1|^2c_2^2+\bar{c_1}^2c_3^2-(c_2^3-2c_1c_2c_3\\
& &+2c_3^2)\bar{c_2}.
\end{eqnarray}
Using~\eqref{k4},~\eqref{k2} and~\eqref{k3}, we can write~\eqref{1} as
\begin{equation}\label{5}
  |M_5^4c_4-A|^2\leq B^2.
\end{equation}
Now we compute $A,\;B$ and $M_5^4$ in terms of $\zeta_i$'s. By substituting the values of $c_1,c_2$ and $c_3$ given in Lemma~\ref{lemmacho} in~\eqref{A} and~\eqref{B}, we respectively obtain
\begin{eqnarray*}
 A&=&16(1-|\zeta_1|^2)^2(1-|\zeta_2|^2)(\zeta_2^2(1-|\zeta_1|^2)^2+\zeta_1^4+(1-|\zeta_1|^2)(3\zeta_1^2\zeta_2-2\zeta_2^2|\zeta_1|^2+\bar{\zeta_1}^2\zeta_2^3)\\
& &+(1-|\zeta_1|^2)(1-|\zeta_2|^2)(2\zeta_1\zeta_3-2\bar{\zeta_1}\zeta_2\zeta_3-\bar{\zeta_2}\zeta_3^2)),\\
  B&=&16(1-|\zeta_1|^2)^3(1-|\zeta_2|^2)^2(1-|\zeta_3|^2)
\end{eqnarray*}
and
\begin{eqnarray*}
  M_5^4 &=& 8 - 4|c_1|^2 + \bar{c_1}^2c_2+ c_1^2\bar{c_2} - 2|c_2|^2 \\
   &=& 8(1-|\zeta_1|^2)^2(1-|\zeta_2|^2).
\end{eqnarray*}
Since $\zeta_i\in\overline{\mathbb{D}},$ we have $B\geq 0.$ Therefore~\eqref{5} yields
\begin{equation*}
  M_5^4c_4=A+B\zeta_4\quad\text{for some}\;\zeta_4\in\overline{\mathbb{D}},
\end{equation*}
which can further be written as
\begin{eqnarray*}
  8(1-|\zeta_1|^2)^2(1-|\zeta_2|^2)c_4 &=& 16(1-|\zeta_1|^2)^2(1-|\zeta_2|^2)(\zeta_2^2(1-|\zeta_1|^2)^2+\zeta_1^4+(1-|\zeta_1|^2)\\
 & & \times(3\zeta_1^2\zeta_2-2\zeta_2^2|\zeta_1|^2+\bar{\zeta_1}^2\zeta_2^3)+(1-|\zeta_1|^2)(1-|\zeta_2|^2)(2\zeta_1\zeta_3\\
 & &-2\bar{\zeta_1}\zeta_2\zeta_3-\bar{\zeta_2}\zeta_3^2)) +16(1-|\zeta_1|^2)^3(1-|\zeta_2|^2)^2(1-|\zeta_3|^2)\zeta_4.\\
\end{eqnarray*}
Equivalently, we may write
\begin{eqnarray*}
&&8(1-|\zeta_1|^2)^2(1-|\zeta_2|^2)(c_4-2(\zeta_1^4+(1-|\zeta_1|^2)(3\zeta_1^2\zeta_2-2\zeta_2^2|\zeta_1|^2+\bar{\zeta_1}^2\zeta_2^3)+\zeta_2^2(1-|\zeta_1|^2)^2\\
&&+(1-|\zeta_1|^2)(1-|\zeta_2|^2)(2\zeta_1\zeta_3-2\bar{\zeta_1}\zeta_2\zeta_3-\bar{\zeta_2}\zeta_3^2+(1-|\zeta_3|^2)\zeta_4)))=0,
\end{eqnarray*}
where $\zeta_i\in\overline{\mathbb{D}}\;(i=1,2,3,4).$ For $\zeta_1=1$ or $\zeta_2=1,$ we have $c_1=2$ and so $p(z)=(1+z)/(1-z)$ for which the result follows obviously. Thus~\eqref{state} follows for $\zeta_1\neq1$ and $\zeta_2\neq1.$
\end{proof}
\textbf{Alternate Proof}
\begin{proof}
  Continuing the proof of \cite[Lemma 2.4]{cho}, we can assume that $b_1,b_1^{(1)},b_1^{(2)}\in\mathbb{D}$. Now let us define
  \begin{equation*}
    \varphi_3(z):=\dfrac{\omega_2(z)}{z},\quad z\in\mathbb{D}\backslash\{0\},\quad \varphi_3(0):=b_1^{(2)}.
  \end{equation*}
Since the function $\varphi_3$ is a self map of $\mathbb{D},$ a function
\begin{equation*}
  \omega_3(z):=\psi_{b_1^{(2)}}(\varphi_3(z))=b_1^{(3)}z+b_2^{(3)}z^2+\cdots,\quad z\in\mathbb{D},
\end{equation*}
is a Schwarz function. By the Schwarz lemma,
\begin{equation*}
  |b_1^{(3)}|=|\omega_3'(0)|=\dfrac{|b_2^{(2)}|}{1-|b_1^{(2)}|^2}\leq 1,
\end{equation*}
that is,
\begin{equation*}
  b_1^{(3)}=\zeta_4\quad\text{for some}\quad \zeta_4\in\overline{\mathbb{D}}.
\end{equation*}
We know that $b_1^{(2)}=\zeta_3$ and therefore,
\begin{equation*}
  b_2^{(2)}=(1-|\zeta_3|^2)\zeta_4.
\end{equation*}
On the other hand
\begin{eqnarray*}
  b_2^{(2)}&=&\dfrac{1}{2}\omega_1''(0)=\dfrac{b_4(1-|b_1|^2)+\bar{b_1}b_2b_3}{(1-|b_1|^2)^2-|b_2|^2}+\dfrac{(b_3(1-|b_1|^2)+b_2^2\bar{b_1})(\bar{b_1}b_2(1-|b_1|^2)+b_3\bar{b_2})}{((1-|b_1|^2)^2-|b_2|^2)^2}\\
  &=&\dfrac{b_4+\zeta_2\bar{\zeta_1}(1-|\zeta_1|^2)(\zeta_3(1-|\zeta_2|^2)-\bar{\zeta_1}\zeta_2^2)}{(1-|\zeta_1|^2)(1-|\zeta_2|^2)}+\zeta_3(\bar{\zeta_1}\zeta_2+\bar{\zeta_2}\zeta_3),
\end{eqnarray*}
which implies that
\begin{equation*}
  b_4=\bar{\zeta_1}^2\zeta_2^3(1-|\zeta_1|^2)-(1-|\zeta_1|^2)(1-|\zeta_2|^2)(2\bar{\zeta_1}\zeta_2\zeta_3+\bar{\zeta_2}\zeta_3^2)+(1-|\zeta_1|^2)(1-|\zeta_2|^2)(1-|\zeta_3|^2)\zeta_4.
\end{equation*}
Also, comparing both sides of $p(z)=(1+\omega(z))/(1-\omega(z))$ we get $c_4=2(b_1^4+3b_1^2 b_2+b_2^2+2b_1b_3+b_4).$ On substituting the values of $b_i$'s, we obtain the result. Moreover equality occurs when
$$\omega(z)=z\psi_{-\zeta_1}(z\psi_{-\zeta_2}(z\psi_{-\zeta_3}(\zeta_4z))),\quad z\in\mathbb{D}.$$
\end{proof}
In the next section, we find sharp upper bounds of the first five consecutive inverse coefficients for functions in each of the classes $\mathcal{F}_1,\;\mathcal{F}_2,\;\mathcal{F}_3$ and $\mathcal{F}_4$. Although the fifth inverse coefficient bound obtained here was not sharp for functions in $\mathcal{F}_2,$ but the range in which the sharp bound lies is also pointed out.
\section{Main Results}
Let $f(z)$ be given by~\eqref{fz} and belongs to $\mathcal{K}_0$. Then we have
\begin{equation}\label{abc}
\RE\left(\dfrac{zf'(z)}{g(z)}\right)>0,
\end{equation}
 where $g$ is a starlike function defined as
\begin{equation*}
  g(z)=z+\sum_{n=2}^{\infty}b_nz^n.
\end{equation*}
In view of~\eqref{abc}, there exists $p\in\mathcal{P}$, with power series representation given by~\eqref{cartheodry}, such that
\begin{equation}\label{compare0}
zf'(z)=g(z)p(z),
\end{equation}
which implies
\begin{equation}\label{compare}
  z+\sum_{n=2}^{\infty}na_nz^n=\left(z+\sum_{n=2}^{\infty}b_nz^n\right)\left(1+\sum_{n=1}^{\infty}c_nz^n\right).
\end{equation}
Upon equating the like term coefficients on either side of~\eqref{compare}, we get
\begin{eqnarray} \label{aibici1}
  2a_2&=&b_2+c_1\\ \label{aibici2}
  3a_3&=&b_3+b_2c_1+c_2\\  \label{aibici3}
  4a_4&=&b_4+c_1b_3+c_2b_2+c_3\\ \label{aibici4}
  5a_5&=&b_5+b_4c_1+b_3c_2+b_2c_3+c_4.
\end{eqnarray}
Using~\eqref{aibici1}-\eqref{aibici4} in~\eqref{deltanai1}-\eqref{deltanai4}, we get the following $\delta_i$'s in terms of $b_i$'s and $c_i$'s:
\begin{eqnarray}\label{delta2}
  \delta_2 &=& -\dfrac{1}{2}(b_2+c_1) \\ \label{delta3}
  \delta_3 &=& \dfrac{1}{6}(3b_2^2+3c_1^2+4b_2c_1-2b_3-2c_2) \\ \label{delta4}
  \nonumber \delta_4 &=& \dfrac{1}{24}(20b_2b_3-25b_2^2c_1+14b_2c_2+14c_1b_3-25b_2c_1^2+20c_1c_2\\
   & &-15b_2^3-15c_1^3-6b_4-6c_3)\\ \label{delta5}
\nonumber  \delta_5 &=& \dfrac{7}{8}b_2^4-\dfrac{7}{4}b_2^2b_3+\dfrac{1}{3}b_3^2+\dfrac{3}{4}b_4b_2-\dfrac{1}{5}b_5+\dfrac{7}{4}b_2^3c_1-\dfrac{25}{12}b_2b_3c_1+\dfrac{11}{20}b_4c_1+\dfrac{25}{12}b_2^2c_1^2\\ \nonumber
  & &-b_3c_1^2+\dfrac{7}{4}b_2c_1^3+\dfrac{7}{8}c_1^4-b_2^2c_2+\dfrac{7}{15}b_3c_2-\dfrac{25}{12}b_2c_1c_2-\dfrac{7}{4}c_1^2c_2+\dfrac{1}{3}c_2^2+\dfrac{11}{20}b_2c_3\\
  & &+\dfrac{3}{4}c_1c_3-\dfrac{1}{5}c_4.
\end{eqnarray}
Applying triangle inequality in~\eqref{delta2}, we obtain
\begin{equation}\label{modelta2}
  2|\delta_2|\leq|b_2|+|c_1|.
\end{equation}
In a similar way, using triangle inequality in~\eqref{delta3} and applying Lemma.\ref{lemma1}, we get
\begin{equation}\label{modelta3}
  3|\delta_3|\leq 2-\dfrac{|c_1|^2}{2}+\left|(b_3-\dfrac{1}{2}b_2^2)-(c_1+b_2)^2\right|.
\end{equation}
Let $c_1=pe^{i\alpha}$ and $q=\cos{\alpha}$ such that $0\leq p\leq 2$ and $0\leq\alpha\leq 2\pi$. Now, we rewrite~\eqref{modelta3} in terms of $p$ and $q$ as follows:
\begin{equation}\label{delta3pq}
3|\delta_3|\leq 2-\dfrac{p^2}{2}+\left|(b_3-\dfrac{1}{2}b_2^2)-\left(pq+ip\sqrt{1-q^2}+b_2\right)^2\right|.
\end{equation}
\begin{theorem}\label{theo1}
  Let $f(z)=z+a_2z^2+a_3z^3+\ldots\in\mathcal{F}_1$. Then $(i)\;|\delta_2|\leq 3/2,$ $(ii)\;|\delta_3|\leq 17/6.$ Further, if $a_2\in\mathbb{R},$ then $(iii)\;|\delta_4|\leq 49/8,$ $(iv)\;|\delta_5|\leq 1729/120.$  These bounds are sharp.
\end{theorem}
\begin{proof}
 (i)\;Let $f\in\mathcal{F}_1$. Since $f$ is close-to-convex function with respect to the starlike function $z/(1-z)$, from~\eqref{compare0} we have
  \begin{equation}\label{2}
  zf'(z)=\dfrac{z}{1-z}p(z).
  \end{equation}
For $g(z)=z/(1-z)$, we have $b_i=1$ for all $i$. Thus in view of Lemma~\ref{lemma3},~\eqref{modelta2} reduces to
\begin{equation*}
  2|\delta_2|\leq 1+|c_1|\leq 3.
\end{equation*}
We know that this inequality is sharp, whenever $|c_1|=2$ which is true for the function $P_{1,\vartheta}(z)\;(0\leq\vartheta\leq2\pi)$ given in Lemma~\ref{lemma1}. The inequality $|\delta_2|\leq3/2$ is sharp since there exists an extremal function $\hat{f}_1\in\mathcal{F}_1,$ which is a solution of $z\hat{f'}_1(z)=z(1-z)^{-1}P_{1,\vartheta}(z)$.\\

\noindent(ii)\;By taking $b_2=b_3=1$ in~\eqref{delta3pq}, we get
\begin{equation*}
 3|\delta_3|\leq 2-\dfrac{p^2}{2}+\left|\dfrac{1}{2}-\left(1+pq+ip\sqrt{1-q^2}\right)^2\right|,
\end{equation*}
which can be written as
\begin{equation*}
  3|\delta_3|\leq 2-\dfrac{p^2}{2}+\sqrt{(p^2+\dfrac{1}{2}+2pq)^2+2p^2(1-q^2)}=:\phi_1(p,q).
\end{equation*}
In the domain $D=\{(p,q):0\leq p\leq 2, -1\leq q\leq 1\}$, we need to find the points where $\phi_1(p,q)$ attains its maximum. A simple computation shows that there exists no solution of
\begin{equation*}
  \dfrac{\partial\phi_1(p,q)}{\partial p}=0\quad\text{and}\quad\dfrac{\partial\phi_1(p,q)}{\partial q}=0
\end{equation*}
in $\mathbb{R}\times\mathbb{R}$. Thus, the maximum value of $\phi_1(p,q)$ is not attained inside $D$. Now, we consider the edges of $D$ in order to find the maximum of $\phi_1(p,q)$. On the line segment $p=0$, $\phi_1(0,q)=5/2.$ On the line segment $p=2$, $\phi_1(2,q)=\sqrt{8q^2+36q+113/4},$ which is an increasing function for $q\in[-1,1].$ Thus $\max_{q\in[-1,1]}\phi_1(2,q)=\phi_1(2,1)=17/2.$ On the line segment $q=-1$, $\phi_1(p,-1)=2-p^2/2+|1/2-2p+p^2|$ which further reduces into two cases. By using elementary calculus, we obtain that $\max_{p\in[0,2]}\phi_1(p,-1)=\phi_1(0,-1)=5/2.$ On the line segment $q=1$, $\phi_1(p,1)=5/2+p^2/2+2p,$ which is an increasing function for $p\in[0,2].$ Thus $\max_{p\in[0,2]}\phi_1(p,1)=\phi_1(2,1)=17/2.$ Hence the maximum value of $\phi_1(p,q)$ is attained at $(2,1)$ and is equal to $17/2$. Thus $|\delta_3|\leq 17/6.$  The inequality $|\delta_3|\leq17/6$ is sharp since there exists an extremal function $\tilde{f}_1\in\mathcal{F}_1,$ which is a solution of $z\tilde{f'}_1(z)=z(1-z)^{-1}P_{1,0}(z)$.\\

\noindent(iii) Upon substituting $b_2=b_3=b_4=1$ in~\eqref{delta4}, we get
\begin{equation*}
 \delta_4=\dfrac{1}{24} \left(-1-11c_1-25c_1^2-15c_1^3+14c_2+20c_1c_2-6c_3\right).
\end{equation*}
We have $a_2\in\mathbb{R},$ which together with~\eqref{aibici1} yields that $c_1\in\mathbb{R}.$ Then by using Lemma~\ref{lemmacho}, we have $\zeta_1\in\mathbb{R}$ and
\begin{eqnarray}\nonumber
24\delta_4 &=& -1-22\zeta_1-72\zeta_1^2-52\zeta_1^3+28(1-\zeta_1^2)\zeta_2+56(1-\zeta_1^2)\zeta_1\zeta_2-12(1-\zeta_1^2)(1-\zeta_2^2)\zeta_3\\ \label{00}
& &+12(1-\zeta_1^2)\zeta_1\zeta_2^2,
\end{eqnarray}
for $\zeta_1\in[-1,1]$ and $\zeta_2,\zeta_3\in\overline{\mathbb{D}}.$ Note that for $\zeta_1=-1$ and $\zeta_1=1$ respectively,
\begin{equation}\label{03}
  \delta_4=1/24\;\text{and}\;\delta_4=-147/24.
\end{equation}
So now we consider $\zeta_1\in(-1,1).$ Clearly from~\eqref{00}, we get
\begin{equation*}
  24|\delta_4|\leq 12(1-\zeta_1^2)\psi(A,B,C,M),
\end{equation*}
where
\begin{equation}\label{psi}
  \psi(A,B,C,M):=|M|(1-|\zeta_2|^2)+|A+B\zeta_2+C\zeta_2^2|
\end{equation} with
\begin{equation*}
  A=\dfrac{1+22\zeta_1+72\zeta_1^2+52\zeta_1^3}{12(1-\zeta_1^2)},\;B=-\dfrac{7}{3}(1+2\zeta_1),\; C=-\zeta_1\;\text{and}\;M=1.
\end{equation*}
Next we observe that for $r_1\approx -0.968128,\;r_2\approx -0.361546$ and $r_3\approx -0.0549415 ,$ we have
\begin{equation*}
  AC\geq 0,\;\zeta_1\in[r_1,r_2]\cup[r_3,0]
  \end{equation*}
  and
  \begin{equation*}
   AC<0,\;\zeta_1\in(-1,r_1)\cup(r_2,r_3)\cup(0,1).
\end{equation*}
\indent Case I. $\zeta_1\in[r_1,r_2]$.  \\
In view of the Lemma~\ref{maxlemma}, we check for the inequality $|B|\geq 2(|M|-|C|)$ and observe that for $\zeta_1\in[r_1,-0.65]$ the inequality holds and $|B|< 2(|M|-|C|)$ for $\zeta_1\in(-0.65,r_2).$ Then by using Lemma~\ref{maxlemma}, for $\zeta_1\in[r_1,-0.65],$ we have $24|\delta_4|\leq 12(1-\zeta_1^2)\max\psi(A,B,C,M)=12(1-\zeta_1^2)(|A|+|B|+|C|).$ A calculation shows that
\begin{equation*}
 12(1-\zeta_1^2)(|A|+|B|+|C|)=-27-46\zeta_1+100\zeta_1^2+120\zeta_1^3=\varphi_1(\zeta_1),
\end{equation*}
where
\begin{equation}\label{phi}
\varphi_1(x):=-27-46x+100x^2+120x^3,\quad x\in[-1,1].
\end{equation} Using elementary calculus, we find that in the interval $[r_1,-0.65],$ $\varphi_1(\zeta_1)$ attains its maximum at $\zeta_1'\approx -0.730479$ and thus $|\delta_4|\leq \varphi_1(\zeta_1')/24\approx 0.5495.$ For the case when $\zeta_1\in(-0.65,r_2),$
\begin{eqnarray*}
  24|\delta_4|&\leq&12(1-\zeta_1^2)\left(|M|+|A|+\dfrac{B^2}{4(|M|-|C|)}\right)\\
  &=&\dfrac{88}{3}+71\zeta_1+60\zeta_1^2-\dfrac{40}{3}\zeta_1^3=\varphi_2(\zeta_1),
\end{eqnarray*}
where
\begin{equation}\label{phi2}
    \varphi_2(x):=\dfrac{88}{3}+71x+60x^2-\dfrac{40}{3}x^3,\quad x\in[-1,1].
\end{equation}
In the interval $(-0.65,r_2),$ the maximum value of the function $\varphi_2(\zeta_1)$ is attained at $\zeta_1=r_2,$ which further implies $|\delta_4|\leq \varphi_2(r_2)/24\approx 0.505693.$\\
\indent Case II. $\zeta_1\in[r_3,0].$\\
For this range of $\zeta_1,$ the inequality $|B|\geq 2(1-|C|)$ holds and thus by Lemma~\ref{maxlemma}, we have
\begin{eqnarray*}
24|\delta_4|&\leq& 12(1-\zeta_1^2) \max\psi(A,B,C,M) \\
&=& 12(1-\zeta_1^2)(|A|+|B|+|C|) \\
   &=& 29+64\zeta_1+44\zeta_1^2+8\zeta_1^3=\varphi_3(\zeta_1),
\end{eqnarray*}
where
\begin{equation}\label{phi3}
    \varphi_3(x):=29+64x+44x^2+8x^3,\quad x\in[-1,1]
\end{equation}
which is an increasing function and hence attains its maximum value at $\zeta_1=0.$ Hence $|\delta_4|\leq \varphi_3(0)/24=29/24\approx 1.20833.$\\
\indent Case III. $\zeta_1\in (-1,r_1).$\\
For this range of $\zeta_1,$ $B^2+4AC(M^2C^{-2}-1)>0$ and $|B|\not> 2(1-|C|),$ so now we check for another set of inequalities given in Lemma~\ref{maxlemma}. We observe that $|AB|\leq |C|(|B|-4|A|)$ when $\zeta_1\in(r_4,r_1),$ where $r_4\approx -0.983158$ and $|C|(|B|+4|A|)>|AB|$ throughout the interval $(-1,r_1).$ So we may conclude that
\begin{equation*}
  \max\psi(A,B,C,M)=(-|A|+|B|+|C|),\quad \zeta_1\in[r_4,r_1)
\end{equation*}
and
\begin{equation*}
    \max\psi(A,B,C,M)=(|C|+|A|)\sqrt{1-\dfrac{B^2}{4AC}},\quad \zeta_1\in(-1,r_4).
\end{equation*}
Hence for $\zeta_1\in[r_4,r_1),$ $|\delta_4|\leq -27-46\zeta_1+100\zeta_1^2+120\zeta_1^3=\varphi_1(\zeta_1),$ where $\varphi_1$ is given by~\eqref{phi}. Note that $\varphi_1(\zeta_1)$ is increasing on this interval and attains is maximum at $r_1.$ So $|\delta_4|\leq\varphi_1(r_1)/24\approx 0.0988825.$ On the other hand, when $\zeta_1\in(-1,r_4)$
\begin{equation*}
    24|\delta_4|\leq (-1-34\zeta_1-72\zeta_1^2-40\zeta_1^3)\sqrt{1+\dfrac{49(1+2\zeta_1)^2(1-\zeta_1^2)}{3\zeta_1(1+22\zeta_1+72\zeta_1^2+52\zeta_1^3)}}=\vartheta(\zeta_1),
\end{equation*}
which is an increasing function of $\zeta_1$ for the specified range of $\zeta_1.$ Therefore it attains its maximum value at $\zeta_1=r_4,$ which further implies that $|\delta_4|\leq \vartheta(r_4)/24\approx 0.0516135.$\\
\indent Case IV. $\zeta_1\in(r_2,r_3).$\\
In view of the Lemma~\ref{maxlemma}, we find out that $B^2C^2\geq -4AC(M^2-C^2)$ in this interval and $|B|<2(1-|C|)$ for $\zeta_1\in(r_2,-0.125).$ So we can say that
\begin{eqnarray*}
    24|\delta_4|&\leq & 12(1-\zeta_1^2)\max\psi(A,B,C,M)\\
    &=&12(1-\zeta_1^2)\left(1-|A|+\dfrac{B^2}{4(|M|-|C|)}\right)\\
    &=&\dfrac{88}{3}+71\zeta_1+60\zeta_1^2-\dfrac{40}{3}\zeta_1^3=\varphi_2(\zeta_1),
\end{eqnarray*}
where $\varphi_2$ is given by~\eqref{phi2}. It is easy to verify that $\varphi_2(\zeta_1)$ increases on the interval $(r_2,-0.125)$ and so $|\delta_4|\leq \varphi_2(-0.125)/25\approx0.0755622.$ Next we observe that for $\zeta_1\in[-0.125,r_3),$ $|C|(|B|-4|A|)\geq |AB|$ and therefore
\begin{eqnarray*}
    24|\delta_4|&\leq& 12(1-\zeta_1^2)(-|A|+|B|+|C|)\\
    &=& 29+66\zeta_1+44\zeta_1^2+8\zeta_1^3=\varphi_3(\zeta_1),
\end{eqnarray*}
where $\varphi_3$ is given by~\eqref{phi3}. For $\zeta_1\in[-0.125,r_3),$ $\varphi_3(\zeta_1)$ increases and thus attains its maximum at $\zeta_1=r_3.$ Hence
   $|\delta_4|\leq \varphi_3(r_3)/24\approx 1.06272.$  \\
\indent Case V. $\zeta_1\in (0,1).$\\
In this interval, we have $B^2>4(|M|+|C|)^2$ and $B^2C^2< -4AC(M^2-C^2).$ Thus $\max\psi(A,B,C,M)=S(A,B,C)=|A|+|B|-|C|$ as $|C|(|B|+4|A|)\leq |AB|$ in this interval. So
\begin{eqnarray*}
    24|\delta_4|&\leq & 12(1-\zeta_1^2)(|A|+|B|-|C|)\\
    &=&29+66\zeta_1+44\zeta_1^2+8\zeta_1^3=\varphi_3(\zeta_1).
\end{eqnarray*}
Clearly $\varphi_3(\zeta_1)$ is increasing on $(0,1)$ and thus $|\delta_4|< \varphi_3(1)/24=49/8.$ Summarizing the inequalities obtained in cases I-V and the bounds obtained in~\eqref{03}, we get $|\delta_4|\leq 49/8.$ The function $\tilde{k}_1\in\mathcal{F}_1,$ obtained by solving~\eqref{2} with $P_{1,0}(z)$ in place of $p(z)$ acts as an extremal function for the inequality $|\delta_4|\leq 49/8$ and hence is sharp.\\
\noindent(iv) Upon substituting $b_2=b_3=b_4=b_5=1$ in~\eqref{delta5}, we get
\begin{equation*}
\begin{aligned}
 \delta_5=&\dfrac{1}{120}+\dfrac{13}{60}c_1+\dfrac{13}{12}c_1^2+\dfrac{7}{4}c_1^3+\dfrac{7}{8}c_1^4-\dfrac{8}{15}c_2-\dfrac{25}{12}c_1c_2-\dfrac{7}{4}c_1^2c_2+\dfrac{1}{3}c_2^2+\dfrac{11}{20}c_3+\dfrac{3}{4}c_1c_3-\dfrac{1}{5}c_4,
\end{aligned}
\end{equation*}
which can be written as
\begin{equation}\label{17}
\begin{aligned}
  \delta_5=\dfrac{1}{5}A+\dfrac{11}{20}B+\varrho(c_1,c_2,c_3)+\varsigma(c_1,c_2),
\end{aligned}
\end{equation}
where
\begin{eqnarray*}
A&=&c_1^4-3c_1^2c_2+c_2^2+2c_1c_3-c_4,\\
B&=&c_3-2c_1c_2+c_1^3,\\
\varrho(c_1,c_2,c_3)&=&\dfrac{1}{120}+\dfrac{13}{60}c_1+\dfrac{2}{15}c_2^2+\dfrac{7}{20}c_1c_3\qquad\text{and}\\
\varsigma(c_1,c_2)&=&\dfrac{8}{15}\left(\dfrac{65}{32}c_1^2-c_2\right)+\dfrac{59}{60}c_1\left(\dfrac{72}{59}c_1^2-c_2\right)+\dfrac{23}{20}c_1^2\left(\dfrac{27}{46}c_1^2-c_2\right).
\end{eqnarray*}
In view of Lemma~\ref{lemma5}, we have $|A|\leq 2$ and $|B|\leq 2.$ Evidently
\begin{equation*}
    |\varrho(c_1,c_2,c_3)|\leq\dfrac{1}{120}+\dfrac{13}{60}|c_1|+\dfrac{2}{15}|c_2|^2+\dfrac{7}{20}|c_1||c_3|\leq \dfrac{19}{8} .
\end{equation*}
Now, let us consider $\varsigma(c_1,c_2)$ which can be again written in terms of $\zeta_1$ and $\zeta_2$ by using Lemma~\ref{lemmacho} as follows:
\begin{equation}\label{18}
  \varsigma_1(\zeta_1,\zeta_2)=\dfrac{1-\zeta_1^2}{15}\left(\dfrac{49\zeta_1^2+85\zeta_1^3+24\zeta_1^4}{1-\zeta_1^2}-16\zeta_2-59\zeta_1\zeta_2-138\zeta_1^2\zeta_2\right).
\end{equation}
Note that for $\zeta_1=-1$ and $\zeta_1=1$ respectively
\begin{equation*}
  \delta_5=-4/5\;\text{and}\;\delta_5=158/15.
\end{equation*}
Clearly, we may write~\eqref{18} as
\begin{equation*}
    \varsigma_1(\zeta_1,\zeta_2)=\dfrac{1-\zeta_1^2}{15}\psi(A,B,C,M),
\end{equation*}
where $\psi(A,B,C,M)$ is given by~\eqref{psi} with
\begin{equation*}
A:=\dfrac{49\zeta_1^2+85\zeta_1^3+24\zeta_1^4}{1-\zeta_1^2},\;B=-16-59\zeta_1-138\zeta_1^2,\;C=0\;\text{and}\;M=0.
\end{equation*}
Since $C=0,$ we consider the first case of the Lemma~\ref{maxlemma} and find out that $|B|\geq 2(|M|-|C|)$ on $(-1,1).$ Thus
\begin{eqnarray*}
 |\varsigma_1(\zeta_1,\zeta_2)|&\leq & \dfrac{1-\zeta_1^2}{15}\max\psi(A,B,C,M)\\
 &=&\dfrac{1-\zeta_1^2}{15}(|A|+|B|+|C|)\\
 &=& -\dfrac{38 \zeta_1^4}{5}+\dfrac{26 \zeta_1^3}{15}+\dfrac{57 \zeta_1^2}{5}+\dfrac{59 \zeta_1}{15}+\dfrac{16}{15}=\vartheta(\zeta_1),
\end{eqnarray*}
where
\begin{equation*}
  \vartheta(x):=-\dfrac{38 x^4}{5}+\dfrac{26 x^3}{15}+\dfrac{57 x^2}{5}+\dfrac{59 x}{15}+\dfrac{16}{15},\quad\quad x\in(-1,1).
\end{equation*}
Using elementary calculus, we find out that $\vartheta(\zeta_1)$ attains its maximum value at $\zeta_1=1$ and thus
$|\varsigma_1(\zeta_1,\zeta_2)|\leq \vartheta(1)=158/15.$
Now applying triangle inequality on~\eqref{17}, we have
\begin{eqnarray*}
  |\delta_5| &\leq & \dfrac{1}{5}|A|+\dfrac{11}{20}|B|+|\varrho(c_1,c_2,c_3)|+|\varsigma(c_1,c_2)| \\
   &\leq & \dfrac{1}{5}(2)+\dfrac{11}{20}(2)+\dfrac{19}{8}+\dfrac{158}{15} \\
   &=& \dfrac{1729}{120}.
\end{eqnarray*}
  The above inequality is sharp since there exists an extremal function $\hat{k}_1\in\mathcal{F}_1,$ which is a solution of $z\hat{k'}_1(z)=z(1-z)^{-1}P_{1,0}(z)$.
\end{proof}
\begin{theorem}
Let $f(z)=z+a_2z^2+a_3z^3+\ldots\in\mathcal{F}_2$. Then $(i)\;|\delta_2|\leq 1.$ Further if $a_2\in\mathbb{R},$ then $(ii)\;|\delta_3|\leq 1,$ $(iii)\;|\delta_4|\leq 16/3\sqrt{15},$ $(iv)\;|\delta_5|\leq 2.947584.$ Except $(iv)$ rest all above bounds are sharp.
\end{theorem}
\begin{proof}
 (i)\;Let $f\in\mathcal{F}_2$. Since $f$ is close-to-convex function with respect to the starlike function $z/(1-z^2)$, we have
\begin{equation*}
  zf'(z)=\dfrac{z}{1-z^2}p(z).
\end{equation*}
If $g(z)=z/(1-z^2)$, then we have $b_2=0,\;b_3=1,\;b_4=0$ and $b_5=1$. Thus, in view of Lemma~\ref{lemma3},~\eqref{modelta2} reduces to
\begin{equation*}
  |\delta_2|\leq \dfrac{|c_1|}{2}\leq 1.
\end{equation*}
We know that this inequality is sharp, whenever $|c_1|=2$ which is true for the function $P_{1,\vartheta}(z)\;(0\leq\vartheta<2\pi)$ given in Lemma~\ref{lemma1}. The upper bound of $|\delta_2|$ is sharp since there exists an extremal function $\hat{f}_2\in\mathcal{F}_2$, which is the solution of $z\hat{f'}_2(z)=z(1-z^2)^{-1}P_{1,\vartheta}(z)\;(0\leq\vartheta<2\pi)$.\\

\noindent(ii)\;By taking $b_2=0\;\text{and}\;b_3=1$ in~\eqref{delta3}, we get
\begin{equation*}
 \delta_3=\dfrac{1}{6}(3c_1^2-2-2c_2).
\end{equation*}
We have $a_2\in\mathbb{R},$ which together with~\eqref{aibici1} yields that $c_1\in\mathbb{R}.$ Then by using Lemma~\ref{lemmacho}, we have $\zeta_1\in\mathbb{R}$ and
\begin{equation*}
  |\delta_3|= \dfrac{1}{6}(8\zeta_1^2-2-4(1-\zeta_1^2)\zeta_2).
\end{equation*}
Using the triangle inequality, we get
\begin{equation*}
  |\delta_3|= \dfrac{1}{6}(|8\zeta_1^2-2|+4(1-\zeta_1^2)|\zeta_2|).
\end{equation*}
For the sake of convenience, we shall take $|\zeta_2|=r.$ Now we have to find the maximum of $\phi_2(\zeta_1,r):=|8\zeta_1^2-2|+4(1-\zeta_1^2)r,$ whenever $-1\leq \zeta_1\leq1$ and $0\leq r\leq 1$. If $\zeta_1\in(-1/2,1/2)$, then $\phi_2(\zeta_1,r)=2-8\zeta_1^2+4r(1-\zeta_1^2)$ attains its maximum value $6$ at $(0,1)$. Further, if $\zeta_1\in[-1,-1/2]\cup[1/2,1]$, then $\phi_2(\zeta_1,r)=8\zeta_1^2-2+4r(1-\zeta_1^2)$ attains its maximum value $6$ at $(1,r)$. Clearly at $\zeta_1=1/2,$ we have $\phi_2(1,r)=3r,$ which obviously attains its maximum value $3$ at $r=1.$ Hence $\max_{r\in[0,1],\zeta_1\in[-1,1]}\phi_2(\zeta_1,r)=6$ and therefore $|\delta_3|\leq1.$ The upper bound of $|\delta_3|$ is sharp since there exists an extremal function $\tilde{f}_2\in\mathcal{F}_2,$ which is the solution of $z\tilde{f}'_2(z)=(z/(1-z^2))P_{0,0}(z).$\\

\noindent(iii) Upon substituting $b_2=0,\;b_3=1,\;b_4=0$ in~\eqref{delta4}, we get
\begin{equation*}
 \delta_4=\dfrac{1}{24} \left(14c_1-15c_1^3+20c_1c_2-6c_3\right).
\end{equation*}
We have $a_2\in\mathbb{R},$ which together with~\eqref{aibici1} yields that $c_1\in\mathbb{R}.$ Then by using Lemma~\ref{lemmacho}, we have $\zeta_1\in\mathbb{R}$ and
\begin{equation}\label{011}
24\delta_4 = 28\zeta_1-52\zeta_1^3+56(1-\zeta_1^2)\zeta_1\zeta_2-12(1-\zeta_1^2)(1-\zeta_2^2)\zeta_3+12(1-\zeta_1^2)\zeta_1\zeta_2^2
\end{equation}
for $\zeta_1\in[-1,1]$ and $\zeta_2,\zeta_3\in\overline{\mathbb{D}}.$ Note that for $\zeta_1=-1$ and $\zeta_1=1$ respectively,
\begin{equation}\label{012}
\delta_4=1\quad \text{and}\quad\delta_4=-1.
\end{equation}
So now we consider $\zeta_1\in(-1,1).$ Clearly from~\eqref{011}, we get
\begin{equation*}
  24|\delta_4|\leq 12(1-\zeta_1^2)\psi(A,B,C,M),
\end{equation*}
where $\psi(A.B,C,M)$ is given by~\eqref{psi} with
\begin{equation*}
  A=\dfrac{-28\zeta_1+52\zeta_1^3}{12(1-\zeta_1^2)},\;B=-\dfrac{14}{3}\zeta_1,\; C=-\zeta_1\;\text{and}\;M=1.
\end{equation*}
Let us set $r_1= -\sqrt{7/13},$ $r_2=\sqrt{7/13},$ $r_3\approx -0.907485$ and $r_4\approx-0.767772.$ Then we have
\begin{equation*}
  AC\geq 0,\;\zeta_1\in[r_1,r_2]
  \end{equation*}
  and
  \begin{equation*}
   AC<0,\;\zeta_1\in(-1,r_1)\cup(r_2,1).
\end{equation*}
Since we are using the same method as used in Theorem~\ref{theo1}(iii), we summarize the above three cases in the following table:
\begin{center}
\renewcommand{\arraystretch}{2}
\begin{tabular}{ |c|c|c|c|c| }
\hline
Cases & Subcases & $\max\psi(A,B,C,M)$  & $12(1-\zeta_1^2)\max\psi(A,B,C,M)$ & Bound \\
\hline
\multirow{4}{*}{$[r_1,r_2]$} & $[r_1,-0.3]$ & $|A|+|B|+|C|$ & $6\zeta_1(5\zeta_1^2-4)$ & $16/(3\sqrt{15})$  \\ \cline{2-5}
& $(-0.3,0)$ & $|M|+|A|+\dfrac{B^2}{4(|M|-|C|)}$ & $-\dfrac{10 \zeta_1^3}{3}+\dfrac{40 \zeta_1^2}{3}-7 \zeta_1+3$ & 1.065\\
\cline{2-5}
& $(0,0.3)$ & $|M|+|A|+\dfrac{B^2}{4(|M|-|C|)}$ & $\dfrac{10 \zeta_1^3}{3}+\dfrac{40 \zeta_1^2}{3}+7 \zeta_1+3$ & 1.065\\
\cline{2-5}
& $[0.3,r_2]$ & $|A|+|B|+|C|$ & $-6\zeta_1(5\zeta_1^2-4)$ & $16/(3\sqrt{15})$  \\ \hline
\multirow{3}{*}{$(-1,r_1)$} & $(-1,r_3]$ & $|A|+|B|-|C|$ & $-2\zeta_1(2+\zeta_1^2)$ & 1  \\ \cline{2-5}
& $(r_3,r_4)$ & $(|C|+|A|)\sqrt{1-\dfrac{B^2}{4AC}}$ & $-\dfrac{2 \zeta_1 \left(5 \zeta_1^2-2\right)}{3 \left(1-\zeta_1^2\right)} \sqrt{\dfrac{28-10 \zeta_1^2}{39 \zeta_1^2-21}}$ & 0.854103  \\ \cline{2-5}
& $[r_4,r_1)$ & $-|A|+|B|+|C|$ & $6\zeta_1(-4+5\zeta_1^2)$ & 0.9595  \\ \hline
\multirow{1}{*}{$(r_2,1)$} & $(r_2,-r_4)$ & $-|A|+|B|+|C|$ & $-6\zeta_1(-4+5\zeta_1^2)$ & 0.9595  \\ \cline{2-5}
& $(-r_4,-r_3)$ & $(|C|+|A|)\sqrt{1-\dfrac{B^2}{4AC}}$ & $\dfrac{2 \zeta_1 \left(5 \zeta_1^2-2\right)}{3 \left(1-\zeta_1^2\right)} \sqrt{\dfrac{28-10 \zeta_1^2}{39 \zeta_1^2-21}}$ & 0.854103  \\ \cline{2-5}
& $(-r_3,1]$ & $|A|+|B|-|C|$ & $2\zeta_1(2+\zeta_1^2))$ & 1  \\ \hline
\end{tabular}
\end{center}
Summing up all the cases mentioned in the above table and the bounds given by~\eqref{012}, we conclude that $|\delta_4|\leq 16/3\sqrt{15}.$ The function $\tilde{k}_2\in\mathcal{F}_2,$ obtained by solving $z\tilde{k}'_2(z)=(z/(1-z^2))H_{t_0,-1}(z),$ where $H_{t,\beta}(z)$ is given by~\eqref{H}, acts as an extremal function for the inequality $|\delta_4|\leq 16/(3\sqrt{15}).$ Hence the inequality qualifies to be sharp.\\

\noindent (iv) Upon substitution of $b_2=0$, $b_3=1$, $b_4=0$ and $b_5=1$ in~\eqref{delta5}, we get
\begin{equation*}
  \delta_5= \dfrac{2}{5}-c_1^2+\dfrac{7}{8}c_1^4+\dfrac{7}{15}c_2-\dfrac{7}{4}c_1^2c_2+\dfrac{1}{3}c_2^2+\dfrac{3}{4}c_1c_3-\dfrac{1}{5}c_4.
\end{equation*}
In view of Lemma~\ref{ownlem3} with $\zeta_1\in[-1,1]$ and $\zeta_i\in\overline{\mathbb{D}}\;(i=2,3,4)$, we have
\begin{eqnarray*}\nonumber
  \delta_5&=&\dfrac{2}{15}-\dfrac{46}{15}\zeta_1^2+\dfrac{59}{15}\zeta_1^4+\dfrac{14}{15}(1-\zeta_1^2)\zeta_2-\dfrac{98}{15}\zeta_1^2\zeta_2(1-\zeta_1^2)-\dfrac{11}{5}\zeta_1^2\zeta_2^2(1-\zeta_1^2)-\dfrac{2}{5}\zeta_1^2\zeta_2^3(1-\zeta_1^2)\\ \nonumber
  & &+\dfrac{14}{15}\zeta_2^2(1-\zeta_1^2)^2+\dfrac{11}{5}\zeta_1\zeta_3(1-\zeta_1^2)(1-|\zeta_2|^2)+\dfrac{4}{5}\zeta_1\zeta_2\zeta_3(1-\zeta_1^2)(1-|\zeta_2|^2)\\
          & &+\dfrac{2}{5}(1-\zeta_1^2)(1-|\zeta_2|^2)\bar{\zeta_2}\zeta_3^2-\dfrac{2}{5}(1-\zeta_1^2)(1-|\zeta_2|^2)(1-|\zeta_3|^2)\zeta_4.
\end{eqnarray*}
Now applying triangle inequality, we get
\begin{equation}\label{17b}
  |\delta_5|\leq |\gamma_1(\zeta_1,\zeta_2)|+|\gamma_2(\zeta_1,\zeta_2)||\zeta_3|+|\gamma_3(\zeta_1,\zeta_2)||\zeta_3|^2+|\gamma_4(\zeta_1,\zeta_2,\zeta_3)||\zeta_4|,
\end{equation}
where
\begin{eqnarray*}
 &\gamma_1(\zeta_1,\zeta_2)=&\dfrac{2}{15}-\dfrac{46}{15}\zeta_1^2+\dfrac{59}{15}\zeta_1^4+\dfrac{14}{15}(1-\zeta_1^2)\zeta_2-\dfrac{98}{15}\zeta_1^2\zeta_2(1-\zeta_1^2)-\dfrac{11}{5}\zeta_1^2\zeta_2^2(1-\zeta_1^2)\\
& & -\dfrac{2}{5}\zeta_1^2\zeta_2^3(1-\zeta_1^2)+\dfrac{14}{15}\zeta_2^2(1-\zeta_1^2)^2\\
  &\gamma_2(\zeta_1,\zeta_2) =& \dfrac{11}{5}\zeta_1(1-\zeta_1^2)(1-|\zeta_2|^2)+\dfrac{4}{5}\zeta_1\zeta_2(1-\zeta_1^2)(1-|\zeta_2|^2) \\
 &\gamma_3(\zeta_1,\zeta_2) =& \dfrac{2}{5}(1-\zeta_1^2)(1-|\zeta_2|^2)\bar{\zeta_2} \\
  &\gamma_4(\zeta_1,\zeta_2,\zeta_3) =& -\dfrac{2}{5}(1-\zeta_1^2)(1-|\zeta_2|^2)(1-|\zeta_3|^2).
\end{eqnarray*}
Taking $\zeta_2=re^{i\theta}$ with $d=\cos{\theta}$ such that $0\leq r\leq1$ and $0\leq\theta\leq2\pi,$ we may write
\begin{equation}\label{18b}
Q(\zeta_1,r,d):=|\gamma_1(\zeta_1,\zeta_2)|^2.
\end{equation}
 Then
\begingroup
\allowdisplaybreaks
\begin{eqnarray*}
  Q(\zeta_1,r,d)&=&\dfrac{944}{75} d^3 r^3 \zeta_1^8-\dfrac{112}{5} d^3 r^3 \zeta_1^6+\dfrac{256}{25} d^3 r^3 \zeta_1^4-\dfrac{32}{75} d^3 r^3 \zeta_1^2+\dfrac{784}{75} d^2 r^4 \zeta_1^8-\dfrac{112}{5} d^2 r^4 \zeta_1^6\\
  & &+\dfrac{336}{25} d^2 r^4 \zeta_1^4-\dfrac{112}{75} d^2 r^4 \zeta_1^2+\dfrac{11092}{225} d^2 r^2 \zeta_1^8-\dfrac{23044}{225} d^2 r^2 \zeta_1^6+\dfrac{1656}{25} d^2 r^2 \zeta_1^4\\
  & &-\dfrac{3064}{225} d^2 r^2 \zeta_1^2+\dfrac{112 d^2 r^2}{225}+\dfrac{188}{75} d r^5 \zeta_1^8-\dfrac {144}{25} d r^5 \zeta_1^6+4 d r^5 \zeta_1^4-\dfrac{56}{75} d r^5 \zeta_1^2\\
  & &+\dfrac{7088}{225} d r^3 \zeta_1^8-\dfrac{18704}{225} d r^3 \zeta_1^6+\dfrac{5332}{75} d r^3 \zeta_1^4-\dfrac{4772}{225} d r^3 \zeta_1^2+\dfrac{392 d r^3}{225}+\dfrac{11564}{225} d r \zeta_1^8\\
  & &-\dfrac{22232}{225} d r \zeta_1^6+\dfrac{1372}{25} d r \zeta_1^4-\dfrac{1736}{225} d r \zeta_1^2+\dfrac{56 d r}{225}+\dfrac{4 r^6 \zeta_1^8}{25}-\dfrac{8 r^6 \zeta_1^6}{25}+\dfrac{4 r^6 \zeta_1^4}{25}\\
  & &+\dfrac{1033 r^4 \zeta_1^8}{225}-\dfrac{3214 r^4 \zeta_1^6}{225}+\dfrac{47 r^4 \zeta_1^4}{3}-\dfrac{308 r^4 \zeta_1^2}{45}+\dfrac{196 r^4}{225}+\dfrac{4058 r^2 \zeta_1^8}{225}-\dfrac{2086 r^2 \zeta_1^6}{45}\\
  & &+\dfrac{2612 r^2 \zeta_1^4}{75}-\dfrac{1604 r^2 \zeta_1^2}{225}+\dfrac{28 r^2}{45}+\dfrac{3481 \zeta_1^8}{225}-\dfrac{5428 \zeta_1^6}{225}+\dfrac{784 \zeta_1^4}{75}-\dfrac{184 \zeta_1^2}{225}+\dfrac{4}{225}.
\end{eqnarray*}
\endgroup
It is sufficient to find the points in the rectangular cube $K:=\{(\zeta_1,r,d):-1\leq \zeta_1\leq 1, 0\leq r\leq 1\;\text{and}\;-1\leq d\leq 1\}$, where the maximum value of $Q(\zeta_1,r,d)$ is attained. In order to find maximum in the interior of $K$, we try to find the points where
\begin{equation*}
  \dfrac{\partial Q(\zeta_1,r,d)}{\partial c}=\dfrac{\partial Q(\zeta_1,r,d)}{\partial r}=\dfrac{\partial Q(\zeta_1,r,d)}{\partial d}=0.
\end{equation*}
After few steps of calculation, we find out that the above set of equations has no solution inside $K.$ Now we are left with six faces namely, $\zeta_1=-1,\;\zeta_1=1,\;r=0,\;r=1,\;d=-1,\;d=1$ and twelve edges of $K$ given by $\zeta_1=-1,\;r=0;\zeta_1=-1,\;r=1;\zeta_1=1,\;r=0;\zeta_1=1,\;r=1;\zeta_1=-1,\;d=-1;\zeta_1=-1,\;d=1;\zeta_1=1,\;d=-1;\zeta_1=1,\;d=1;r=0,\;d=-1;r=0,\;d=1;r=1,\;d=-1;r=1,\;d=1.$ In all these cases, we use elementary techniques to find maximum values and conclude that $Q(\zeta_1,r,d)$ attains its maximum at the point $(\sqrt{15/7},1,0)$ and is equal to $791/392.$ Next, let us write $|\zeta_2|=r$ and $|\zeta_3|=q$ in
\begin{equation}\label{19b}
 G(\zeta_1,r,q):= |\gamma_2(\zeta_1,\zeta_2)||\zeta_3|+|\gamma_3(\zeta_1,\zeta_2)||\zeta_3|^2+|\gamma_4(\zeta_1,\zeta_2,\zeta_3)|.
\end{equation}
Again using the same method, we find the maximum of  $G(c,r,q)$ on the cuboid $\{(\zeta_1,r,q):-1\leq \zeta_1\leq 1,0\leq r\leq 1,0\leq q\leq 1\}$ and observe that the maximum value is attained at $(\zeta_1',r',1)$, where $\zeta_1'\approx 1.12539$ is the smallest root of $5552+1936\zeta_1-6200\zeta_1^2-2776\zeta-1^3+1479\zeta_1^4+945\zeta_1^5=0$ and $r'=11(4-3(\zeta_1')^2)/(4(3(\zeta_1')^2+2(\zeta_1')-4))\approx0.268895$. Hence $G(\zeta_1',r',1)\approx 0.929727.$ From~\eqref{17b},~\eqref{18b} and~\eqref{19b}, we obtain
\begin{equation*}
|\delta_5|\leq  \max_{\zeta_1\in[-1,1],r\in[0,1],s\in[-1,1]}\sqrt{Q(\zeta_1,r,d)}+\max_{\zeta_1\in[-1,1],r\in[0,1],q\in[0,1]}G(\zeta_1,r,q).
\end{equation*}
Hence $|\delta_5|\leq 791/392+0.929727\approx 2.947584$. However, this bound is not sharp but we may conclude that the sharp bound lies in the range $[791/392, 791/392 + 0.929727]$ as there exists a function $\hat{k}_2(z)\in\mathcal{F}_2$ for which the fifth inverse coefficient is equal to $791/392.$ The function $\hat{k}_2(z)$ can be obtained by solving $z\hat{k}_2(z)=(z/(1-z^2))H_{t_1,\beta_1}(z)$, where $t_1=(1/56)\left(14-\sqrt{105}\right)$ and $\beta=-1.$
\end{proof}
\begin{theorem}
Let $f(z)=z+a_2z^2+a_3z^3+\ldots\in\mathcal{F}_3$. Then $(i)\;|\delta_2|\leq 3/2,$ $(ii)\;|\delta_3|\leq 19/6.$ Further, if $a_2\in\mathbb{R},$ then $(iii)\;|\delta_4|\leq 61/8,$ $(iv)\;|\delta_5|\leq 2371/120.$  These bounds are sharp.
\end{theorem}
\begin{proof}
 (i)\;Let $f\in\mathcal{F}_3$. Since $f$ is close-to-convex function with respect to the starlike function $z/(1-z+z^2),$ from~\eqref{compare0} we have
  \begin{equation}\label{2c}
  zf'(z)=\dfrac{z}{1-z+z^2}p(z).
  \end{equation}
For $g(z)=z/(1-z+z^2)$, we have $b_2=1,\;b_3=0,\;b_4=-1$ and $b_5=-1$. Thus, in view of Lemma~\ref{lemma3},~\eqref{modelta2} reduces to
\begin{equation*}
  2|\delta_2|\leq 1+|c_1|\leq 3.
\end{equation*}
We know that this inequality is sharp whenever $|c_1|=2,$ which is true for the function $P_{1,\vartheta}(z)\;(0\leq\vartheta<2\pi),$ defined in Lemma~\ref{lemma1}. The inequality $|\delta_2|\leq3/2$ is sharp since there exists an extremal function $\hat{f}_3\in\mathcal{F}_3,$ which is a solution of $z\hat{f'}_3(z)=z(1-z)^{-1}P_{1,\vartheta}(z)$.\\

\noindent(ii)\;By taking $b_2=1$ and $b_3=0$ in~\eqref{delta3pq}, we get
\begin{equation*}
 3|\delta_3|\leq 2-\dfrac{p^2}{2}+\left|-\dfrac{1}{2}-\left(1+pq+ip\sqrt{1-q^2}\right)^2\right|,
\end{equation*}
which can be written as
\begin{equation*}
  3|\delta_3|\leq 2-\dfrac{p^2}{2}+\sqrt{(p^2+\dfrac{3}{2}+2pq)^2-2p^2(1-q^2)}=:\phi_3(p,q).
\end{equation*}
In the domain $D=\{(p,q):0\leq p\leq 2, -1\leq q\leq 1\}$, we need to find the points where $\phi_3(p,q)$ attains its maximum. A simple computation shows that the only solution of
\begin{equation*}
  \dfrac{\partial \phi_3(p,q)}{\partial p}=0\quad\text{and}\quad\dfrac{\partial \phi_3(p,q)}{\partial q}=0
\end{equation*}
in $\mathbb{R}\times\mathbb{R}$ is (0,0). Thus, the maximum value of $\phi_3(p,q)$ is not attained inside $D$. Now, we consider the edges of $D$ to find the maximum of $\phi_3(p,q)$. On the line segment $p=0$, $\phi_3(0,q)=7/2.$ On the line segment $p=2$, $\phi_3(2,q)=\sqrt{(11/2+4q)^2-8(1-q^2)}$ which is a decreasing function of $q,$ whenever $-1\leq q< -11/12$ and thus maximum value is attained at $q=-1$. Similarly, $\phi_3(2,q)$ is an increasing function of $q,$ whenever $-11/12<q\leq 1$ and thus maximum value is attained at $q=1$. Now, we have $\phi_3(2,-1)=3/2$ and $\phi_3(2,1)=19/2$. Thus, $\max_{q\in[-1,1]}\phi_3(2,q)=\phi_3(2,1)=19/2.$ On the line segment $q=-1$, $\phi_3(p,-1)=7/2+p^2/2-2p$ which is a decreasing function of $p$ whenever $0\leq p\leq2.$ Thus, $\max_{p\in[0,2]}\phi_3(p,-1)=\phi_3(0,-1)=7/2.$ On the line segment $q=1$,  $\phi_3(p,1)=7/2+p^2/2+2p$ which is an increasing function of $p$ whenever $0\leq p\leq2.$ Thus, $\max_{p\in[0,2]}\phi_3(p,1)=\phi_3(2,1)=19/2.$ It is easy to note that maximum value of $\phi_3(p,q)$ is attained at $(2,1)$ and is equal to $19/2$. Thus $|\delta_3|\leq 19/6.$ The extremal function $\tilde{f}_3\in\mathcal{F}_3$ for which the upper bound of $|\delta_3|$ is sharp, can be obtained by solving~\eqref{2c} with $P_{1,0}(z)$ in place of $p(z)$.\\

\noindent(iii) Upon substituting $b_2=1,\;b_3=0,\;b_4=-1$ in~\eqref{delta4}, we get
\begin{equation*}
 \delta_4=\dfrac{1}{24} \left(-9-25c_1-25c_1^2-15c_1^3+14c_2+20c_1c_2-6c_3\right).
\end{equation*}
We have $a_2\in\mathbb{R},$ which together with~\eqref{aibici1} yields that $c_1\in\mathbb{R}.$ Then by using Lemma~\ref{lemmacho}, we have $\zeta_1\in\mathbb{R}$ and
\begin{eqnarray}\nonumber
24\delta_4 &=& -9-50\zeta_1-72\zeta_1^2-52\zeta_1^3+28(1-\zeta_1^2)\zeta_2+56(1-\zeta_1^2)\zeta_1\zeta_2-12(1-\zeta_1^2)(1-\zeta_2^2)\zeta_3\\ \label{02}
& &+12(1-\zeta_1^2)\zeta_1\zeta_2^2,
\end{eqnarray}
for $\zeta_1\in[-1,1]$ and $\zeta_2,\zeta_3\in\overline{\mathbb{D}}.$ Note that for $\zeta_1=-1$ and $\zeta_1=1$ respectively,
\begin{equation}\label{04}
\delta_4=\dfrac{21}{24}\quad \text{and}\quad\delta_4=\dfrac{61}{8}.
\end{equation}
So now we consider $\zeta_1\in(-1,1).$ Clearly from~\eqref{02}, we get
\begin{equation*}
  24|\delta_4|\leq 12(1-\zeta_1^2)\psi(A,B,C,M),
\end{equation*}
where $\psi(A,B,C,M)$ is given by~\eqref{psi} with
\begin{equation*}
  A=\dfrac{9+50\zeta_1+72\zeta_1^2+52\zeta_1^3}{12(1-\zeta_1^2)},\;B=-\dfrac{7}{3}(1+2\zeta_1),\; C=-\zeta_1\;\text{and}\;M=1.
\end{equation*}
Let us set $r_1\approx -0.257982$ and $r_2\approx -0.29465.$ Then we have
\begin{equation*}
  AC\geq 0,\;\zeta_1\in[r_1,0]
  \end{equation*}
  and
  \begin{equation*}
   AC<0,\;\zeta_1\in(-1,r_1)\cup(0,1).
\end{equation*}
Since we are using the same method as used in Theorem~\ref{theo1}(iii), we summarize the above three cases in the following table:
\begin{center}
\renewcommand{\arraystretch}{2}
\begin{tabular}{ |c|c|c|c|c| }
\hline
Cases & Subcases & $\max\psi(A,B,C,M)$  & $12(1-\zeta_1^2)\max\psi(A,B,C,M)$ & Bound \\
\hline
\multirow{2}{*}{$[r_1,0]$} & $[r_1,-0.125)$ & $|M|+|A|+\dfrac{B^2}{4(|M|-|C|)}$ & $\dfrac{112}{3}+99\zeta_1+60\zeta_1^2-\dfrac{40}{3}\zeta_1^3$ & 1.08008  \\ \cline{2-5}
& $[-0.125,0]$ & $|A|+|B|+|C|$ & $37+94\zeta_1+44\zeta_1^2+8\zeta_1^3$ & 37/24 \\
\hline
\multirow{2}{*}{$(-1,r_1)$} & $(-1,r_2)$ & $1+|A|+\dfrac{B^2}{4(|M|+|C|)}$ & $\dfrac{1}{3}(58+95\zeta_1+140\zeta_1^2+40\zeta_1^3)$ & 7/8  \\ \cline{2-5}
& $[r_2,r_1)$ & $1-|A|+\dfrac{B^2}{4(|M|-|C|)}$ & $\dfrac{112}{3}+99\zeta_1+60\zeta_1^2-\dfrac{40}{3}\zeta_1^3$ & 0.585036  \\ \hline
\multirow{1}{*}{$(0,1)$} & $-$ & $|A|+|B|-|C|$ & $37+94\zeta_1+44\zeta_1^2+8\zeta_1^3$ & 61/8  \\ \hline
\end{tabular}
\end{center}
Summing up all the cases mentioned in the above table and the bounds given by~\eqref{04}, we conclude that $|\delta_4|\leq 61/8.$ The function $\tilde{k}_3\in\mathcal{F}_3$ obtained by solving $z\tilde{k}'_3(z)=z(1-z+z^2)^{-1}P_{1,0}(z)$ acts as an extremal function for the inequality $|\delta_4|\leq 61/8$ and hence the inequality becomes sharp.\\

\noindent(iv)\;Upon substituting $b_2=1,\;b_3=0,\;b_4=-1$ and $b_5=-1$ in~\eqref{delta5}, we get
\begin{equation*}
\begin{aligned}
 \delta_5=&\dfrac{13}{40}+\dfrac{6}{5}c_1+\dfrac{25}{12}c_1^2+\dfrac{7}{4}c_1^3+\dfrac{7}{8}c_1^4-c_2-\dfrac{25}{12}c_1c_2-\dfrac{7}{4}c_1^2c_2+\dfrac{1}{3}c_2^2+\dfrac{11}{20}c_3+\dfrac{3}{4}c_1c_3-\dfrac{1}{5}c_4,
\end{aligned}
\end{equation*}
which can be written as
\begin{equation}\label{17c}
\begin{aligned}
  \delta_5=\dfrac{1}{5}A+\dfrac{11}{20}B+\varrho(c_1,c_2,c_3)+\varsigma(c_1,c_2),
\end{aligned}
\end{equation}
where
\begin{eqnarray*}
A&=&c_1^4-3c_1^2c_2+c_2^2+2c_1c_3-c_4,\\
B&=&c_3-2c_1c_2+c_1^3,\\
\varrho(c_1,c_2,c_3)&=&\dfrac{13}{40}+\dfrac{6}{5}c_1+\dfrac{2}{15}c_2^2+\dfrac{7}{20}c_1c_3\qquad\text{and}\\
\varsigma(c_1,c_2)&=&\left(\dfrac{25}{12}c_1^2-c_2\right)+\dfrac{59}{60}c_1\left(\dfrac{72}{59}c_1^2-c_2\right)+\dfrac{23}{20}c_1^2\left(\dfrac{27}{46}c_1^2-c_2\right).
\end{eqnarray*}
Using Lemma~\ref{lemma5}, we have $|A|\leq 2$ and $|B|\leq 2.$ Evidently
\begin{equation*}
  |\varrho(c_1,c_2,c_3)|\leq \dfrac{13}{40}+\dfrac{6}{5}|c_1|+\dfrac{2}{15}|c_2|^2+\dfrac{7}{20}|c_1c_3|\leq \dfrac{559}{120}.
\end{equation*}
Now, let us consider $\varsigma(c_1,c_2)$ which can be again written in terms of $\zeta_1$ and $\zeta_2$ by using Lemma~\ref{lemmacho} as follows:
\begin{equation}\label{05}
  \varsigma_1(\zeta_1,\zeta_2)=\dfrac{1-\zeta_1^2}{15}\left(\dfrac{95\zeta_1^2+85\zeta_1^3+24\zeta_1^4}{1-\zeta_1^2}-30\zeta_2-59\zeta_1\zeta_2-138\zeta_1^2\zeta_2\right).
\end{equation}
Note that for $\zeta_1=-1$ and $\zeta_1=1$ respectively
\begin{equation*}
  \delta_5=34/15\quad\text{and}\quad\delta_5=68/5.
\end{equation*}
We may write~\eqref{05} as
\begin{equation*}
    \varsigma_1(\zeta_1,\zeta_2)=\dfrac{1-\zeta_1^2}{15}\psi(A,B,C,M),
\end{equation*}
where $\psi(A,B,C,M)$ is given by~\eqref{psi} with
\begin{equation*}
A=\dfrac{95\zeta_1^2+85\zeta_1^3+24\zeta_1^4}{1-\zeta_1^2},\;B=-30-59\zeta_1-138\zeta_1^2,\;C=0\;\text{and}\;M=0.
\end{equation*}
Since $C=0,$ we consider the first case of the Lemma~\ref{maxlemma} and find out that $|B|\geq 2(|M|-|C|)$ on $(-1,1).$ Thus
\begin{eqnarray*}
 |\varsigma_1(\zeta_1,\zeta_2)|&\leq & \dfrac{1-\zeta_1^2}{15}\max\psi(A,B,C,M)\\
 &=&\dfrac{1-\zeta_1^2}{15}(|A|+|B|+|C|)\\
 &=& -\dfrac{38 \zeta_1^4}{5}+\dfrac{26 \zeta_1^3}{15}+\dfrac{203 \zeta_1^2}{5}+\dfrac{59 \zeta_1}{15}+2=:\vartheta(\zeta_1),
\end{eqnarray*}
where
\begin{equation*}
  \vartheta(x):=-\dfrac{38 x^4}{5}+\dfrac{26 x^3}{15}+\dfrac{203 x^2}{15}+\dfrac{59 x}{15}+2,\quad\quad x\in(-1,1).
\end{equation*}
Using elementary calculus, we find out that $\vartheta(\zeta_1)$ attains its maximum value at $\zeta_1=1$ and thus
$|\varsigma_1(\zeta_1,\zeta_2)|\leq \vartheta(1)=68/5.$
Applying triangle inequality on~\eqref{17c}, we have
\begin{eqnarray*}
  |\delta_5| &\leq & \dfrac{1}{5}|A|+\dfrac{11}{20}|B|+|\varrho(c_1,c_2,c_3)|+|\varsigma(c_1,c_2)| \\
   &\leq & \dfrac{1}{5}(2)+\dfrac{11}{20}(2)+\dfrac{559}{120}+\dfrac{68}{5} \\
   &=& \dfrac{2371}{120}.
\end{eqnarray*}
 The inequality $|\delta_5|\leq 2371/120$ is sharp since there exists an extremal function $\hat{k}_3\in\mathcal{F}_3,$ which is a solution of $z\hat{k}'_3(z)=z(1-z+z^2)^{-1}P_{1,0}(z)$.
\end{proof}
\begin{theorem}
Let $f(z)=z+a_2z^2+a_3z^3+\ldots\in\mathcal{F}_4$. Then $(i)\;|\delta_2|\leq 2,$ $(ii)\;|\delta_3|\leq 5.$ Further, if $a_2\in\mathbb{R},$ then $(iii)\;|\delta_4|\leq 14,$ $(iv)\;|\delta_5|\leq 42.$  These bounds are sharp.
\end{theorem}
\begin{proof}
 (i)\;Let $f\in\mathcal{F}_4$. Since $f$ is close-to-convex function with respect to the starlike function $z/(1-z)^2,$ from~\eqref{compare0} we have
  \begin{equation}\label{2d}
  zf'(z)=\dfrac{z}{(1-z)^2}p(z).
  \end{equation}
For $g(z)=z/(1-z)^2$, we have $b_2=2,\;b_3=3,\;b_4=4$ and $b_5=5$. Thus, in view of Lemma~\ref{lemma3},~\eqref{modelta2} reduces to
\begin{equation*}
  2|\delta_2|\leq 2+|c_1|\leq 4.
\end{equation*}
We know that this inequality is sharp whenever $|c_1|=2,$ which is true for the function $P_{1,\vartheta}(z)\;(0\leq\vartheta<2\pi),$ defined in Lemma~\ref{lemma1}. The inequality $|\delta_2|\leq 2$ is sharp since there exists an extremal function $\hat{f}_4\in\mathcal{F}_4,$ which is a solution of $z\hat{f'}_4(z)=z(1-z)^{-1}P_{1,\vartheta}(z)$.\\

\noindent(ii)\;By taking $b_2=2$ and $b_3=3$ in~\eqref{delta3pq}, we get
\begin{equation*}
 3|\delta_3|\leq 2-\dfrac{p^2}{2}+\left|1-(2+pq+ip\sqrt{1-q^2})^2\right|,
\end{equation*}
which can be written as
\begin{equation*}
  3|\delta_3|\leq 2-\dfrac{p^2}{2}+\sqrt{(p^2+3+4pq)^2+4p^2(1-q^2)}=:\phi_4(p,q).
\end{equation*}
In the domain $D=\{(p,q):0\leq p\leq 2, -1\leq q\leq 1\}$, we need to find the points where $\phi_4(p,q)$ attains its maximum. A simple computation shows that the only solution of
\begin{equation*}
  \dfrac{\partial \phi_4(p,q)}{\partial p}=0\quad\text{and}\quad\dfrac{\partial \phi_4(p,q)}{\partial q}=0
\end{equation*}
in $\mathbb{R}\times\mathbb{R}$ is (0,0). Thus, the maximum value of $\phi_4(p,q)$ is not attained inside $D$. Now, we consider the edges of $D$ to find the maximum of $\phi_4(p,q)$. On the line segment $p=0$, $\phi_4(0,q)=5.$ On the line segment $p=2$, $\phi_4(2,q)=\sqrt{16 \left(1-q^2\right)+(8 q+7)^2}$ which is an increasing function of $q$ and thus maximum value is attained at $q=1$. Now, we have $\phi_4(2,1)=15.$ Thus, we have $\max_{q\in[-1,1]}\phi_4(2,q)=\phi_4(2,1)=15.$ On the line segment $q=-1$, $\phi_4(p,-1)=2-p^2/2+|3-4p+p^2|.$ Using elementary calculus, we obtain that $\phi_4(p,-1)$ attains its maximum value $5$ at $p=0.$ On the line segment $q=1$,  $\phi_4(p,1)=p^2/2+4 p+5$ which is an increasing function of $p$ whenever $0\leq p\leq2.$ Thus, $\max_{p\in[0,2]}\phi_4(p,1)=\phi_4(2,1)=15.$ It is easy to note that maximum value of $\phi_4(p,q)$ is attained at $(2,1)$ and is equal to $15$. Thus $|\delta_3|\leq 5.$ The extremal function $\tilde{f}_4\in\mathcal{F}_4$ for which the upper bound of $|\delta_3|$ is sharp, can be obtained by solving~\eqref{2d} with $P_{1,0}(z)$ in place of $p(z)$.\\

\noindent(iii) Upon substituting $b_2=2,\;b_3=3,\;b_4=4$ in~\eqref{delta4}, we get
\begin{equation*}
 \delta_4=\dfrac{1}{24} \left(-24-58c_1-50c_1^2-15c_1^3+28c_2+20c_1c_2-6c_3\right).
\end{equation*}
We have $a_2\in\mathbb{R},$ which together with~\eqref{aibici1} yields that $c_1\in\mathbb{R}.$ Then by using Lemma~\ref{lemmacho}, we have $\zeta_1\in\mathbb{R}$ and
\begin{eqnarray}\nonumber
24\delta_4 &=& -24-116\zeta_1-144\zeta_1^2-52\zeta_1^3+56(1-\zeta_1^2)\zeta_2+56(1-\zeta_1^2)\zeta_1\zeta_2-12(1-\zeta_1^2)(1-\zeta_2^2)\zeta_3\\ \label{07}
& &+12(1-\zeta_1^2)\zeta_1\zeta_2^2,
\end{eqnarray}
for $\zeta_1\in[-1,1]$ and $\zeta_2,\zeta_3\in\overline{\mathbb{D}}.$ Note that for $\zeta_1=-1$ and $\zeta_1=1$ respectively,
\begin{equation}\label{08}
\delta_4=0\quad \text{and}\quad\delta_4=14.
\end{equation}
So now we consider $\zeta_1\in(-1,1).$ Clearly from~\eqref{07}, we get
\begin{equation*}
  24|\delta_4|\leq 12(1-\zeta_1^2)\psi(A,B,C,M),
\end{equation*}
where $\psi$ is given by~\eqref{psi} with
\begin{equation*}
  A=\dfrac{24+116\zeta_1+144\zeta_1^2+52\zeta_1^3}{12(1-\zeta_1^2)},\;B=-\dfrac{14}{3}(1+2\zeta_1),\; C=-\zeta_1\;\text{and}\;M=1.
\end{equation*}
Let us set $r_1\approx -0.318042,$ $r_2\approx -0.67332$ and $r_3\approx -0.395298.$ Then we have
\begin{equation*}
  AC\geq 0,\;\zeta_1\in[r_1,0]
  \end{equation*}
  and
  \begin{equation*}
   AC<0,\;\zeta_1\in(-1,r_1)\cup(0,1).
\end{equation*}
Since we are using the same method as used in Theorem~\ref{theo1}(iii), we summarize the above three cases in the following table:
{\small
\begin{center}
\renewcommand{\arraystretch}{2}
\begin{tabular}{ |c|c|c|c|c| }
\hline
Cases & Subcases & $\max\psi(A,B,C,M)$  & $12(1-\zeta_1^2)\max\psi(A,B,C,M)$ & Bound \\
\hline
\multirow{1}{*}{$[r_1,0]$} & - & $|A|+|B|+|C|$ & $8(1+\zeta_1)(10+10\zeta_1+\zeta_1^2)$ & 10/3  \\
\hline
\multirow{3}{*}{$(-1,r_1)$} & $(-1,r_2)$ & $1+|A|+\dfrac{B^2}{4(|M|+|C|)}$ & $\dfrac{40}{3}(1+\zeta_1)(4+2\zeta_1+\zeta_1^2)$ & 0.563835  \\ \cline{2-5}
& $[r_2,r_3)$ & $(|C|+|A|)\sqrt{1-\dfrac{B^2}{4AC}}$ & $-\dfrac{2 \left(5 \zeta_1^2+13 \zeta_1+3\right)}{3 (1-\zeta_1)} \sqrt{\dfrac{-10 \zeta_1^3+20 \zeta_1^2+67 \zeta_1+49}{39 \zeta_1^3+69 \zeta_1^2+18 \zeta_1}}$ & 4.30598  \\ \cline{2-5}
& $[r_3,r_1)$ & $-|A|+|B|+|C|$ & $8(1+\zeta_1)(10+10\zeta_1+\zeta_1^2)$ & 1.57322  \\ \hline
\multirow{1}{*}{$(0,1)$} & $-$ & $|A|+|B|-|C|$ & $8(1+\zeta_1)(10+10\zeta_1+\zeta_1^2)$ & 14  \\ \hline
\end{tabular}
\end{center}
}
Summing up all the cases mentioned in the above table and the bounds given by~\eqref{08}, we conclude that $|\delta_4|\leq 61/8.$ The function $\tilde{k}_4\in\mathcal{F}_4$ obtained by solving $z\tilde{k}'_4(z)=z(1-z)^{-2}P_{1,0}(z)$ acts as an extremal function for the inequality $|\delta_4|\leq 14$ and hence the inequality becomes sharp.\\

\noindent(iv)\;Upon substituting $b_2=2,\;b_3=3,\;b_4=4$ and $b_5=5$ in~\eqref{delta5}, we get
\begin{equation*}
\begin{aligned}
 \delta_5=&1+\dfrac{37}{10}c_1+\dfrac{16}{3}c_1^2+\dfrac{7}{2}c_1^3+\dfrac{7}{8}c_1^4-\dfrac{13}{5}c_2-\dfrac{25}{6}c_1c_2-\dfrac{7}{4}c_1^2c_2+\dfrac{1}{3}c_2^2+\dfrac{11}{10}c_3+\dfrac{3}{4}c_1c_3-\dfrac{1}{5}c_4,
\end{aligned}
\end{equation*}
which can be written as
\begin{equation}\label{17d}
\begin{aligned}
  \delta_5=\dfrac{1}{5}A+\dfrac{11}{10}B+\varrho(c_1,c_2,c_3)+\varsigma(c_1,c_2),
\end{aligned}
\end{equation}
where
\begin{eqnarray*}
A&=&c_1^4-3c_1^2c_2+c_2^2+2c_1c_3-c_4,\\
B&=&c_3-2c_1c_2+c_1^3,\\
\varrho(c_1,c_2,c_3)&=&1+\dfrac{37}{10}c_1+\dfrac{2}{15}c_2^2+\dfrac{7}{20}c_1c_3\qquad\text{and}\\
\varsigma(c_1,c_2)&=&\dfrac{13}{5}\left(\dfrac{80}{39}c_1^2-c_2\right)+\dfrac{59}{30}c_1\left(\dfrac{72}{59}c_1^2-c_2\right)+\dfrac{23}{20}c_1^2\left(\dfrac{27}{46}c_1^2-c_2\right).
\end{eqnarray*}
Using Lemma~\ref{lemma5}, we have $|A|\leq 2$ and $|B|\leq 2.$ It can be easily noted that
\begin{equation*}
  |\varrho(c_1,c_2,c_3)|\leq 1+\dfrac{37}{10}|c_1|+\dfrac{2}{15}|c_2|^2+\dfrac{7}{20}|c_1||c_3|\leq \dfrac{31}{3}.
\end{equation*}
Now, let us consider $\varsigma(c_1,c_2)$ which can be again written in terms of $\zeta_1$ and $\zeta_2$ by using Lemma~\ref{lemmacho} as follows:
\begin{equation}\label{09}
  \varsigma_1(\zeta_1,\zeta_2)=\dfrac{2(1-\zeta_1^2)}{15}\left(\dfrac{121\zeta_1^2+85\zeta_1^3+12\zeta_1^4}{1-\zeta_1^2}-39\zeta_2-59\zeta_1\zeta_2-69\zeta_1^2\zeta_2\right).
\end{equation}
Note that for $\zeta_1=-1$ and $\zeta_1=1$ respectively
\begin{equation*}
  \delta_5=34/15\quad\text{and}\quad\delta_5=43/15.
\end{equation*}
It is clear that~\eqref{09} can be written as
\begin{equation*}
    \varsigma_1(\zeta_1,\zeta_2)=\dfrac{2(1-\zeta_1^2)}{15}\psi(A,B,C,M),
\end{equation*}
where $\psi(A,B,C,M)$ is given by~\eqref{psi} with
\begin{equation*}
A=\dfrac{121\zeta_1^2+85\zeta_1^3+12\zeta_1^4}{1-\zeta_1^2},\;B=-39-59\zeta_1-69\zeta_1^2,\;C=0\;\text{and}\;M=0.
\end{equation*}
Since $C=0,$ we consider the first case of the Lemma~\ref{maxlemma} and find out that $|B|\geq 2(|M|-|C|)$ on $(-1,1).$ Thus
\begin{eqnarray*}
 |\varsigma_1(\zeta_1,\zeta_2)|&\leq & \dfrac{1-\zeta_1^2}{15}\max\psi(A,B,C,M)\\
 &=&\dfrac{2(1-\zeta_1^2)}{15}(|A|+|B|+|C|)\\
 &=& \dfrac{2}{15} \left(-57 \zeta_1^4+26 \zeta_1^3+151 \zeta_1^2+59 \zeta_1+39\right)=\vartheta(\zeta_1),
\end{eqnarray*}
where
\begin{equation*}
  \vartheta(x):=\dfrac{2}{15} \left(-57 x^4+26 x^3+151 x^2+59 x+39\right),\quad\quad x\in(-1,1).
\end{equation*}
Using elementary calculus, we find out that $\vartheta(\zeta_1)$ attains its maximum value at $\zeta_1=1$ and thus
$|\varsigma_1(\zeta_1,\zeta_2)|\leq \vartheta(1)=436/15.$
Applying triangle inequality on~\eqref{17d}, we have
\begin{eqnarray*}
  |\delta_5| &\leq & \dfrac{1}{5}|A|+\dfrac{11}{20}|B|+|\bar{\varrho}(c_1,c_2,c_3)|+|\bar{\varsigma}(c_1,c_2)| \\
  &\leq & \dfrac{1}{5}(2)+\dfrac{11}{10}(2)+\dfrac{31}{3}+\dfrac{436}{15} \\
   &=& \dfrac{630}{15}.
\end{eqnarray*}
 The inequality $|\delta_5|\leq 630/15$ is sharp since there exists an extremal function $\hat{k}_4\in\mathcal{F}_4,$ which is a solution of $z\hat{k}'_4(z)=z(1-z)^{-2}P_{1,0}(z)$.
\end{proof}
\section{Concluding Remarks}
\noindent Finding the initial coefficient bounds of functions belonging to a chosen class is a usual phenomenon but attracts special attention if some computational intricacies are resolved in achieving sharp results.  Further, the method or technique underuse matters a lot for further investigations in such problems. In the present work, the estimation of sharp bound for the fourth inverse coefficient gives rise to a maximization problem involving two complex and one real variable, which are eliminated using the maximization technique given by Choi et al.~\cite{max} and used by many other authors. On the other hand, for the fifth inverse coefficient, we obtain an expression involving four variables. The fifth inverse coefficient's sharp bound is mostly determined by splitting the expression and using the suitable  Carath\'{e}odary coefficient bounds results and the maximization method mentioned in the above technique. Thus, the bounds obtained with this technique are all sharp except the fifth inverse coefficient for functions in $\mathcal{F}_2.$ To overcome this, we derived a general formula of the fourth Carath\'{e}odary coefficient, but by using it, we could only improve the earlier bound, obtained by the previous technique used in other cases, for the fifth coefficient of functions in $\mathcal{F}_2$. Further, we found the range in which the sharp fifth inverse coefficient bound falls for functions in $\mathcal{F}_2$; however, locating the sharp bound is still open.

\end{document}